\documentclass[12pt]{article}
\usepackage[labelformat=empty]{caption}

\usepackage{amsmath,amssymb}
\usepackage{amsthm}
\usepackage{graphicx,tikz}
\usepackage{enumerate}

\numberwithin{equation}{section}
\newtheorem{Thm}{Theorem}[section]
\newtheorem{Prop}[Thm]{Proposition}
\newtheorem{Lem}[Thm]{Lemma}
\newtheorem{Cor}[Thm]{Corollary}
\newtheorem{Fact}[Thm]{Fact}

\theoremstyle{remark}

\newtheorem{Rem}[Thm]{Remark}

\theoremstyle{definition}

\newtheorem{Def}[Thm]{Definition}
\newtheorem*{Def*}{Definition}

\numberwithin{equation}{section}

\newcommand{\g}[1]{{\mbox{\goth #1}}}
\newcommand{\m}[1]{\mathbb{ #1}}
\newcommand{\mc}[1]{\mathcal{ #1}}
\newcommand{\gs}[1]{{\mbox{\gots #1}}}

\newcommand{\Ind}{{\operatorname{Ind}}}

     \def\ol{\overline}    
   \def\wt{\widetilde}
\def\al{\alpha}               

       \def\la{\lambda}      
\def\si{\sigma}                
               
              \def\De{\Delta}
             \def\Ph{\Phi}

\newcommand{\rmd}{{\,\rm d}}

\theoremstyle{definition}

\theoremstyle{remark}

\newtheorem{Rmq}[Thm]{Remark}

\newtheorem{conjecture}[Thm]{Conjecture}
\newtheorem{question}[Thm]{Question}

\numberwithin{equation}{section}
\newfont{\goth}{eufm10 at 12pt}
\newfont{\gots}{eufm8 at 9pt}

\def\bt{\begin{Thm}}
\def\et{\end{Thm}}
\def\br{\begin{Rmq}}
\def\er{\end{Rmq}}

\def\bc{\begin{Cor}}
\def\ec{\end{Cor}}
\def\bp{\begin{Prop}}
\def\ep{\end{Prop}}
\def\bl{\begin{Lem}}
\def\el{\end{Lem}}
\def\bd{\begin{Def}}
\def\ed{\end{Def}}
\def\bq{\begin{quotation}}
\def\eq{\end{quotation}}
\def\bfa{\begin{Fact}}
\def\efa{\end{Fact}}

\def\ra{\rightarrow}
\def\vs{\vspace{1em}}

\begin{document}

\title{
Tempered homogeneous spaces IV
}

\author{Yves Benoist and Toshiyuki Kobayashi
}
\date{}
\maketitle


\begin{abstract}{
\noindent Let $G$ be a complex semisimple Lie group 
and $H$ a complex closed connected subgroup.
Let   $\g g$ and  $\g h$ be their Lie algebras.
We prove that  the  regular representation  of $G$
in $L^2(G/H)$ is tempered if and only 
if the ortho\-gonal of $\g h$ in $\g g$
contains regular elements
by showing simultaneously the equivalence
 to other striking conditions
 such as $\g h$ has a solvable limit algebra.  
}\end{abstract}


{\footnotesize \tableofcontents}


\section{Introduction}
\label{secintrod}

Let $X=G/H$ be a homogeneous space of a Lie group $G$.  
This article 
 is the fourth one in our series of papers \cite{BeKoI, BeKoII, BeKoIII} dealing with a question
 about when $L^2(X)$ is tempered,
{\it{i.e.}},  to be weakly contained in the regular representation in $L^2(G)$.
We proved in \cite{BeKoI, BeKoII}
 a criterion \eqref{eqntemrho} below by an analytic and dynamical approach
 when $G$ is real reductive, 
 and accomplished in \cite{BeKoIII} a classification
 of all the pairs $(G,H)$
 of real reductive Lie groups
 for which $L^2(X)$ is non-tempered.  
We refer to the introduction of both \cite{BeKoI} and \cite{BeKoII}
 for some motivations and perspectives on this question.

In this article 
 we find a striking relationship 
 of this question with other
disciplines such as a topological condition
 concerning the \lq\lq{limit subalgebras}\rq\rq\
 and a geometric condition concerning coadjoint orbits.  
The relationship is perfect 
 when $G$ is complex reductive
 (Theorem \ref{thmequcon}).  
For the proof, 
 we explore the temperedness of $L^2(X)$
 beyond reductive setting
 (Theorem \ref{thm:1.1}).  

\subsection{Real homogeneous spaces}
\label{secintrea}

\bq
We extend the criterion in \cite{BeKoI, BeKoII}
 for the temperedness of $L^2(X)$
 to the general setting 
 where $X$ is a homogeneous of a {\it{real}} Lie group
 which is {\it{not}} necessarily reductive.  
\eq

In the first two papers \cite{BeKoI} and \cite{BeKoII}, 
we first noticed that the property of $L^2(G/H)$
 being tempered depends only on the pair $(\g g, \g h)$
 of Lie algebras, 
and introduced 
for an $\g h$-module $V$ and $Y\in \g h$, the quantity
\begin{eqnarray*}
\rho_V(Y)&:=&\mbox{
half the sum of the absolute values of the }\\ 
&&
\mbox{ real part of the eigenvalues of $Y$ in $V$.}
\end{eqnarray*}
We found the following temperedness criterion
 when $G$ is a connected semisimple Lie group with finite center, 
 and $H$ is a connected closed subgroup:
\begin{equation}
\label{eqntemrho}
 L^2(G/{H})\; \mbox{\rm is tempered}
\;\;\;\Longleftrightarrow \;\;\;
\rho_{\gs h}\leq \rho_{\gs g/\gs h} \mbox{\rm{ on ${\mathfrak{h}}$.}  }
\end{equation}
This  criterion \eqref{eqntemrho} 
was proven in  \cite{BeKoI}
 when $\g h$ is assumed to be semisimple
 by a dynamical approach, 
 and was extended in \cite{BeKoII}
 to arbitrary ${\mathfrak{h}}$
 by an idea of \lq\lq{domination of $G$-spaces}\rq\rq.  
Developing the techniques in a more general setting, 
 we extend \eqref{eqntemrho} without any reductive assumptions
 of ${\mathfrak{g}}$ and ${\mathfrak{h}}$:
\begin{Thm}
[see Theorem \ref{thmgh}]
\label{thm:1.1}
Let $G$ be a real algebraic Lie group, 
 and $H$ an algebraic subgroup.  
We fix maximal reductive subgroups $G_{\operatorname{s}}$
 and $H_{\operatorname{s}}$
 of $G$ and $H$, 
 respectively, 
 such that $H_{\operatorname{s}} \subset G_{\operatorname{s}}$.  
Then one has the equivalence:
\[
  \text{$L^2(G/H)$ is $G_{\operatorname{s}}$-tempered
$\Leftrightarrow$
 $\rho_{\mathfrak{g}_{\operatorname{s}}} \le 2 \rho_{\mathfrak{g}/\mathfrak{h}}$ on ${\mathfrak{h}}_{\operatorname{s}}$.  }
\]
\end{Thm}

Theorem \ref{thm:1.1} (and its further generalization 
 to the Hilbert bundle valued case)
 serves as a \lq\lq{tool}\rq\rq\
 in proving the relationship
 with other discipines, 
 which is formulated in Theorem \ref{thmequcon} below.  

\subsection{Temperedness condition and the orbit philosophy}
\bq
We discuss what the orbit philosophy 
suggests about the geometry of coadjoint orbits
 \lq\lq{corresponding to}\rq\rq\
 the temperedness condition of $L^2(G/H)$.  
\eq

Let ${\mathfrak{g}}$ be the Lie algebra 
 of a Lie group $G$, 
 and ${\mathfrak{g}}^{\ast}$ its dual.  
We denote by $\widehat G$ the unitary dual of $G$, 
 {\it{i.e.}}, 
 the set of equivalence classes
 of irreducible unitary representations of $G$.  
The orbit philosophy due to Kirillov--Kostant--Duflo
 expects an intimate connection of the unitary dual $\widehat G$
 with the set of coadjoint orbits
 ${\mathfrak{g}}^{\ast}/\operatorname{Ad}^{\ast}(G)$.  
This works perfectly 
 for simply connected nilpotent groups, 
 but does not exactly for semisimple Lie groups.  
Nevertheless, 
 ${\mathfrak{g}}^{\ast}/\operatorname{Ad}^{\ast}(G)$
 may be considered to be a fairly good approximation
 as a parameter set of $\widehat G$.  
As an expected functionality, 
 the orbit philosophy also suggests
 that the disintegration of $L^2(G/H)$ would be supported
 on the subset of $\widehat G$
 \lq\lq{corresponding to}\rq\rq\ the closure of $\operatorname{Ad}^{\ast}
(G){\mathfrak{h}}^{\perp}/\operatorname{Ad}^{\ast}(G)$
 where $\g h^{\perp}:= \operatorname{Ker}(\g g^{\ast} \to \g h^{\ast})$.  
On the other hand, 
 for a connected semisimple Lie group $G$, 
 loosely speaking, 
 irreducible tempered representations
 of $G$ 
 are supposed to be obtained 
 as \lq\lq{geometric quantization}\rq\rq\
 of semisimple coadjoint orbits
 having amenable isotropy subgroups.  
Thus one expects
 that the temperedness of the unitary representation $L^2(G/H)$ may be characterized
 by its \lq\lq{classical limit}\rq\rq\
 in the geometry of coadjoint orbits
 via the orbit philosophy.  
When $G$ is a complex Lie group, 
 we formulate a precise criterion below from this viewpoint.

\subsection{Complex homogeneous spaces}
\label{secintcom}
\bq
In the third paper \cite{BeKoIII} and in this one, 
we extend and deepen the theory of tempered homogeneous spaces
with focus on the complex setting.  
\eq
Suppose $\g g$ is a complex semisimple algebra.  
Via the Killing form
$$K(X,Y):={\rm tr}({\rm ad}X\,{\rm ad}Y), $$
we identify $\g g^{\ast}$ with $\g g$, 
 and $\g h^{\perp}$ with the orthogonal subspace of $\g h$ in $\g g$
 with respect to $K$.  
An element $X \in \g g$ is called {\it{regular}}
 if its centralizer $\g z_{\gs g}(X)$ in $\g g$  has minimal dimension, 
 {\it{i.e.}},  
$\dim\g z_{\gs g}(X)={\rm rank}\,\g g$.  
We denote by $\g g_{\rm reg}$ the set of regular elements $X$
of $\g g$, 
 and set
$$
\g h^\bot_{\rm reg}:=\g h^\bot\cap\g g_{\rm reg}.
$$
In the third paper \cite{BeKoIII}
 we found yet another but more geometric tempered criterion for $L^2(G/H)$
 when both ${\mathfrak{g}}$ and ${\mathfrak{h}}$ are assumed
 to be complex semisimple Lie algebras.  
As we see in Proposition \ref{proequcon}
 this geometric criterion can be reformulated
 as $\g h^\bot_{\rm reg}\neq\emptyset$. 
In the present paper, 
we extend this criterion
to all complex Lie subalgebras $\g h$ of $\g g$.

\bt
\label{thmtemorb}
Let $\g g$ be a complex semisimple Lie algebra and $\g h$ be a  complex Lie subalgebra.  
Then one has the equivalence~: 
\begin{equation}
\label{eqntemorb}
L^2(G/{{H}})\; \mbox{\rm is tempered}
\;\;\;\Longleftrightarrow \;\;\;
\g h^\bot_{\rm reg}\neq\emptyset.
\end{equation}
\et

Since the set $\g h^\bot_{\rm reg}$ is Zariski open in
$\g h^\bot$, one always has the equivalence
\begin{equation}
\label{eqnempden}
\g h^\bot_{\rm reg}\neq\emptyset
\;\;\;\Longleftrightarrow \;\;\;
\g h^\bot_{\rm reg}\;\;
\mbox{is dense in $\g h^\bot$}, 
\end{equation}
and thus Theorem \ref{thmtemorb}
 fits well into the aforementioned orbit philosophy.

One sees from \cite[Cor.~5.6]{BeKoII}
 that Theorem \ref{thmtemorb} for complex Lie groups
 yields the sufficiency of the temperedness in the real setting as well:
\begin{Cor}
\label{cor:realOrb}
Let $G$ be a real semisimple algebraic Lie group 
 and $H$ an algebraic subgroup.  
If $\g h_{\rm{reg}}^{\perp}\ne \emptyset$, 
 then $L^2 (G/H)$ is tempered.  
\end{Cor}

\begin{Rem}
(1)\enspace
The implications $\Longrightarrow$ in 
\eqref{eqntemorb} and \eqref{eqnrhoorb}
are not always true for a real semisimple Lie group $G$.
For instance, when $G$ is not $\m R$-split 
and $H$ is a maximal compact subgroup,
the representation $L^2(G/H)$ is tempered
but  $\g h^\bot_{\rm reg}$ is empty.
Another example is given by $G/H=SL(3, \m H)/SL(2,\m H)$.
\par\noindent
(2)\enspace
Let ${\mathfrak{g}}_{\operatorname{ame}}$
 denote the set of elements in ${\mathfrak{g}}$
 with amenable stabilizer 
 for the adjoint action of $G$.  
For reductive $H$, 
by \cite[Thm.~1.5]{BeKoIII} and Lemma \ref{lemcoaorb} below, 
 one has the implication:
\begin{equation}
\label{eqn:tempame}
\text{$L^2(G/H)$ is tempered
$\Longrightarrow$ $\g h^{\perp} \cap \g g_{\operatorname{ame}}$ is dense in $\g h^{\perp}$. }
\end{equation}
The converse implication \eqref{eqn:tempame} does not always hold
 even for semisimple symmetric spaces
 (\cite[Sect.~8.5]{BeKoIII}). 
\end{Rem}

By \eqref{eqntemrho}, 
 our main task for Theorem \ref{thmtemorb} will be to prove the following.  

\bp
\label{prorhoorb}
Let $\g g$ be a complex semisimple Lie algebra and $\g h$ a complex Lie subalgebra.  
Then one has the equivalence~: 
\begin{equation}
\label{eqnrhoorb}
2 \rho_{\gs h}\leq \rho_{\gs g}
\;\;\;\Longleftrightarrow \;\;\;
\g h^\bot_{\rm reg}\neq\emptyset.
\end{equation}
\ep

\subsection{The equivalent conditions}
\label{secequcon}

\bq
We now introduce two  other conditions 
that we will prove to be equivalent to \eqref{eqnrhoorb}. 
\eq

Let us think of $\g h$ as a point in the variety $\mc L$ of all Lie subalgebras of $\g g$.  
One surprising feature of the equivalence  \eqref{eqnrhoorb} is
that the left-hand side is a closed condition on $\g h$
while the right-hand side is an open condition on $\g h$.
Since both conditions are invariant by conjugation by $G$,
this remark suggests us to work 
 with the adjoint orbit closure of $\g h$. 
As we will see, 
 this new point of view will be very fruitful,
first by suggesting new striking conditions equivalent to \eqref{eqnrhoorb}
and  eventually by leading to a proof of \eqref{eqnrhoorb}.

Let ${\rm Ad} G$ be the adjoint group, 
let ${\rm Ad} G\,\g h$ be the ${\rm Ad}G$-orbit of $\g h$ in $\mc L$
and $\ol{{\rm Ad} G\,\g h}$ be the closure of this orbit. 
We introduce also the following two $G$-invariant algebraic subvarieties of $\mc L$:
\begin{eqnarray*}
\label{eqnlll}
\mc L_{sol}&:=&
\{\g r\in \mc L\mid \g r \;\;\mbox{\rm is solvable}\},\\	
\mc L_{mun}&:=&
\{\g n\in \mc L\mid \g n  \;\;\mbox{\rm is maximal unipotent in $\g g$}\}.	
\end{eqnarray*}	
We recall that a Lie subalgebra is said to be unipotent if 
all its elements are nilpotent.

As we mentioned,
 we will prove the equivalence \eqref{eqnrhoorb} by showing simultaneously 
the equivalence to other striking conditions that we introduce now.  
Let $H$ be the closure of a connected subgroup of $G$
 with Lie subalgebra $\g h$.  
\begin{eqnarray*}
\label{eqnequcon}
Tem(\g g,\g h)& :&
L^2(G/{{H}})\; \mbox{ is tempered},\\
Rho(\g g,\g h)&:&
\rho_{\gs h}\leq \rho_{\gs g/\gs h},\\
Sla(\g g,\g h)&:&  \ol{{\rm Ad} G\,\g h}\cap \mc L_{sol}\neq\emptyset,\\
Tmu(\g g,\g h)&:& \mbox{there exists\; $ \g n\in \mc L_{mun}$
 such that $\g h\cap \g n=\{ 0\}$},\\
Orb(\g g,\g h)&:&
\g h^\bot_{\rm reg}\neq\emptyset.
\end{eqnarray*}
To refer to these conditions, we might say informally that\\ 
- $\g h$ is a tempered Lie subalgebra,\\
- $\g h$ satisfies the $\rho$-inequality,\\
- $\g h$ admits a solvable limit algebra,\\
- $\g h$ has a transversal maximal unipotent,\\
- $\g h^\bot$ meets a regular orbit. 

\bt
\label{thmequcon}
Let $\g g$ be a complex semisimple Lie algebra and $\g h$ a complex Lie subalgebra.  
Then the following five conditions are equivalent~: 
\begin{equation*}
\label{eqntemrhoslatmuorb1}
Tem(\g g,\g h)
\Longleftrightarrow Rho(\g g,\g h)
\Longleftrightarrow Sla(\g g,\g h)
\Longleftrightarrow Tmu(\g g,\g h)
\Longleftrightarrow Orb(\g g,\g h).
\end{equation*}
\et

The proof of Theorem \ref{thmequcon} will last up to Section \ref{secindrho}. 
\bc
\label{corequcon2}
Let $\g g$ be a complex semisimple Lie algebra. The set $\mc L_{sla}$ of Lie subalgebras $\g h\subset \g g$ satisfying $Sla(\g g,\g h)$ is both closed and open in $\mc L$. 
\ec

\begin{proof}[Proof of Corollary \ref{corequcon2}]  
The condition $Rho(\g g,\g h)$ is closed, 
 while the condition $Orb(\g g,\g h)$
is open.
\end{proof}

\bc
\label{corequcon3}
Let $\g g$ be a complex semisimple Lie algebra and $\g h$ a complex Lie subalgebra. 
Choose $\g h'\in\ol{{\rm Ad} G\,\g h}$. Then one has the equivalence
\begin{equation}
\label{eqn:slaeq}
Sla(\g g,\g h)
\Longleftrightarrow
Sla(\g g,\g h').  
\end{equation}
\ec

\begin{proof}[Proof of Corollary \ref{corequcon3}]  
This is a consequence of Corollary \ref{corequcon2}
\end{proof}

The equivalence \eqref{eqn:slaeq} can be reformulated as follows:
\begin{eqnarray} 
\label{eqnaghclo}
\begin{minipage}[c]{12cm}
{\sl{
If the orbit closure $\ol{{\rm Ad} G\,\g h}$ contains at least one solvable $\g h''$,\\
then any $\g h'$ in $\ol{{\rm Ad} G\,\g h}$ is solvable
 as far as ${\rm Ad} G\,\g h'$ is closed.
}}
\end{minipage}& &
\end{eqnarray}	

Although the statement \eqref{eqn:slaeq} is 
 purely a structure theorem of Lie subalgebras, 
 our proof of \eqref{eqn:slaeq} relies on the theory 
 of unitary representations
 via Theorem \ref{thmequcon}.  
We would like to point out that we are not aware of a more direct proof
 of \eqref{eqn:slaeq}.

\begin{Rem}
We will explain in Theorem \ref{thmtemrhosla}, 
how to extend the equivalence 
$
Tem(\g g,\g h)
\Longleftrightarrow Rho(\g g,\g h)
\Longleftrightarrow Sla(\g g,\g h)
$
to complex algebraic {\it{non-semisimple}} Lie algebras $\g g$. 
In particular, we will see in Corollary \ref{corghghgh} 
that the equivalence \eqref{eqn:slaeq}
is true for any pair $\g g\supset\g h$ of complex Lie algebras.
\end{Rem}

\subsection{Strategy of proof and organization}
\label{secstrpro}

We now explain the strategy of the proof of Theorem \ref{thmequcon}.
Since we already know from \eqref{eqntemrho} the equivalence 
\begin{equation}
\label{eqntemrho0}
Tem(\g g,\g h)\Longleftrightarrow Rho(\g g,\g h)\, ,
\end{equation}
it remains to prove the equivalences 
\begin{equation}
\label{eqnrhoslatmuorb}
Rho(\g g,\g h)
\Longleftrightarrow Sla(\g g,\g h)
\Longleftrightarrow Tmu(\g g,\g h)
\Longleftrightarrow Orb(\g g,\g h).
\end{equation}
All these statements are purely algebraic and we will prove these implications
by algebraic methods in Chapter \ref{secslatmuorb} except for the implication
\begin{equation}
\label{eqnslarho}
 Sla(\g g,\g h)
\Longrightarrow Rho(\g g,\g h).
\end{equation}
The proof of this implication \eqref{eqnslarho} is more delicate
and will be given in Chapter \ref{seccomalg}.  
It will use an induction argument
that reduces to the case where $\g h$ is semisimple.
The induction argument will involve unitary representation theory
 and a parabolic subgroup 
$G_0$ of $G$ containing $H$. This will force us to 
deal with algebraic groups $G$ which are not semisimple.
 
The proof will also use the analytic interpretation
of $Rho(\g g,\g h)$ as a temperedness criterion, 
and the disintegration of the unitary representation $L^2(G_0/H)$.
Indeed we will  spend Chapters \ref{secreaalg} and \ref{secdomuni} 
proving the extension of the  temperedness criterion \eqref{eqntemrho}
that we need. 
This extension (Theorem \ref{thm:1.1}) is valid 
for any real algebraic Lie group $G$ and any real algebraic subgroup $H$.
The proof of this extension will rely on the Hertz majoration principle 
for unitary representations.
\vs

{\bf Acknowledgments.}
The authors are grateful to the IHES and to The University of Tokyo for their support.  
The second author was partially supported 
 by JSPS Kakenhi Grant Number JP18H03669.

\section{Sla, Tmu and Orb}
\label{secslatmuorb}
In this chapter, we focus on the proof of the implications 
in \eqref{eqnrhoslatmuorb} that uses only
algebraic tools. That is all of them except for the implication
\eqref{eqnslarho}.


\subsection{Sla and Tmu}
\label{secslatmu}
\bq
We begin with the easiest of all these equivalences.
\eq
\bp
\label{proslatmu}
Let $\g g$ be a complex semisimple Lie algebra and $\g h\subset \g g$
be a complex Lie subalgebra. Then, one has the equivalence
\begin{eqnarray}
Sla(\g g,\g h)
&\Longleftrightarrow&
Tmu(\g g, \g h)\, .
\end{eqnarray}
\ep

\begin{proof}[Proof of Proposition \ref{proslatmu}]
$\Longrightarrow$  Since we assume $Sla(\g g,\g h)$, 
there exists a se\-quence $(g_n)_{n\geq 1}$ in $G$ such that the limit
$\g r=\lim\limits_{n\ra\infty}{\rm Ad} g_n\,\g h$ 
exists and is a solvable Lie sub\-alge\-bra of 
$\g g$. Since $\g r$ is solvable, 
there exists a Borel subalge\-bra $\g b^-$ of $\g g$ containing $\g r$. Let $\g n$ 
be a maximal unipotent subalgebra of $\g g$ which is opposite to $\g b^-$, so that one has $\g b^-\oplus \g n=\g g$. 
In particular, one has $\g r\cap \g n=\{0\}$ and, for $n$ large,
${\rm Ad}g_n\,\g h\cap \g n=\{0\}$.
This proves $Tmu(\g g,\g h)$.

$\Longleftarrow$ Since we assume $Tmu(\g g,\g h)$, there exists a
maximal unipotent sub\-alge\-bra $\g n$ of $\g g$ 
such that $\g h\cap \g n=\{0\}$. Let $\g b$ be the Borel subalgebra containing $\g n$, let $\g j$ be a Cartan subalgebra
of $\g b$ so that $\g b=\g j\oplus \g n$ and let 
$\g n^-$ be the maximal unipotent subalgebra of $\g g$
which is opposite to $\g b$ and normalized by $\g j$. 
Let $\De=\De(\g g, \g j)$ be the root system of $\g j$
in $\g g$. We write $\De=\De^+ \cup \De^-$ where $\De^+$ 
and $\De^-$ are respectively the roots of $\g j$ in 
$\g n$ and $\g n^-$.
Choose an element $X\in \g j$ in the positive Weyl chamber,
this means that for all $\al\in \De^+$, one has 
${\bf Re}(\al(X))>0$.
Since $\g h\cap \g n=\{0\}$, 
the limit $\g r:=\lim\limits_{n\ra\infty}{\rm Ad} e^{-nX}\,\g h$
exists and is a subalgebra of $\g b^-$. In particular, 
this Lie algebra $\g r$ is solvable.
This proves $Sla(\g g,\g h)$.
\end{proof}

\bc
\label{corslaope}
Let $\g g$ be a complex semisimple Lie algebra. Then, 
the set of subalgebras $\g h$ satisfying 
$Sla(\g g,\g h)$ is open in $\mc L$.
\ec

\begin{proof}
The condition $Tmu(\g g,\g h)$ is clearly an open condition.  
\end{proof}

\subsection{Related Lie subalgebras}
\label{sechhh}
\bq
We now explain why we can often assume that $\g h=[\g h,\g h]$.
\eq

\bl
\label{lemhhh}
Let $\g g$ be a complex semisimple Lie algebra and $\g h\subset \g g$
be a complex Lie subalgebra. Let $G$ be a Lie group with Lie algebra
$\g g$ and $H_1=H$ be the smallest closed subgroup of $G$
whose Lie algebra contains $\g h$.
Set $\g h_0=[\g h,\g h]$ and $\g h_1:=Lie(H)$.
Then, one has the equivalences
\begin{eqnarray}
(i)\;\;\;Sla(\g g,\g h)
&\Longleftrightarrow&
Sla(\g g, \g h_0)\, .\\
(ii)\;\; Sla(\g g,\g h)
&\Longleftrightarrow&
Sla(\g g, \g h_1)\, .
\end{eqnarray}
\el

\begin{proof}[Proof of Lemma \ref{lemhhh}]
$(i)\Longrightarrow$ This follows from the inclusion 
$\g h_0\subset \g h$. 

$(i)\Longleftarrow$ Since we assume
$Sla(\g g,\g h_0)$, there exists a sequence $(g_n)_{n\geq 1}$ in $G$ such that the limit
$\g r_0=\lim\limits_{n\ra\infty}{\rm Ad} g_n\,\g h_0$ 
exists and is a solvable Lie subalgebra of 
$\g g$. Then, after extraction, 
the limit
$\g r:=\lim\limits_{n\ra\infty}{\rm Ad} g_n\,\g h$ 
exists and satisfies 
$[\g r,\g r]\subset \lim \limits_{n\ra\infty}
[{\rm Ad} g_n\,\g h,{\rm Ad} g_n\,\g h]= \g r_0$.
In particular, the limit $\g r$ is a solvable Lie subalgebra of 
$\g g$.
This proves $Sla(\g g,\g h)$.

$(ii)$ This follows from $(i)$ and the inclusions 
$[\g h_1,\g h_1]\subset \g h\subset \g h_1$.
\end{proof}

\subsection{Sla and Orb}
\label{secslaorb}
\bq
The proof of the following equivalence is still purely algebraic but slightly more tricky.
\eq
\bp
\label{proslaorb}
Let $\g g$ be a complex semisimple Lie algebra and $\g h\subset \g g$
be a complex Lie subalgebra. Then, one has the equivalence
\begin{eqnarray}
Sla(\g g,\g h)
&\Longleftrightarrow&
Orb(\g g, \g h)\, .
\end{eqnarray}
\ep

\begin{proof}[Proof of the implication $\Longrightarrow$ in Proposition \ref{proslaorb}]
Since we assume $Sla(\g g,\g h)$, there exists a se\-quence $(g_n)_{n\geq 1}$ in $G$ such that the limit
$\g r=\lim\limits_{n\ra\infty}{\rm Ad} g_n\,\g h$ 
exists and is a solvable Lie subalgebra of 
$\g g$. Since $\g r$ is solvable, there exists a Borel subalge\-bra $\g b$ of $\g g$ containing $\g r$.	Since the orthogonal of $\g b$ 
is the maximal unipotent subalgebra $\g b^\bot=\g n:=[\g b,\g b]$,
the orthogonal $\g r^\bot$ also contains $\g n$.
By a result of Dynkin (see \cite[Thm.~4.1.6]{CoMcG}), the Lie algebra
$\g n$ always contains regular elements of $\g g$, 
the orthogonal $\g r^\bot$ also contains regular elements of $\g g$.
Since the set $\g g_{reg}$ is open, for $n$ large, the orthogonal
${\rm Ad}g_n\,\g h^\bot$ contains regular elements and 
$\g h^\bot$ too. This proves $Orb(\g g, \g h)$.
\end{proof}

The proof of the converse implication will rely on the following two lemmas.

\bl
\label{lemxxlpxu}
Let $\g g$ be a complex semisimple Lie algebra and $\g q=\g l\oplus\g u$
be a parabolic subalgebra where $\g l$ is a reductive Lie subalgebra and $\g u$ is the unipotent
radical of $\g q$. 

Let $X=X_{\gs l}+X_{\gs u}$ be an element of $\g q$ with $X_{\gs l}\in \g l$ 
and $X_{\gs u}\in \g u$. If $X$ is regular in $\g g$, then $X_{\gs l}$ is regular in $\g l$.
\el

Let $r$ be the rank of $\g g$. We recall that the set $\g g_{reg}$ of regular elements
of $\g g$ is the set of elements $X\in \g g$ whose centralizer in $\g g$ has dimension 
$\dim \g z_{\gs g}(X) = r$.
Similarly, the set $\g l_{reg}$ of regular element
of $\g l$ is the set of elements $X\in \g l$ 
whose centralizer in $\g l$ has dimension 
$\dim \g z_{\gs l}(X) = r$.
This set may not be equal to $\g l\cap \g g_{reg}$. 
For instance, when $\g q$ is a Borel subalgebra, 
then $\g l$ is a Cartan subalgebra of $\g g$ and one has 
$\g l_{reg}=\g l$.

\begin{proof}[Proof of Lemma \ref{lemxxlpxu}]
One computes
\begin{eqnarray*}
\dim \g g -r	
&=&
\dim {\rm Ad}G\, X\\
&\leq&
\dim G/Q + \dim {\rm Ad}Q\, X\\
&\leq&
2\dim \g u + \dim ({\rm Ad}Q\, X +\g u)/\g u\\
&=&
2\dim \g u + \dim {\rm Ad}L\, X_{\gs l}.
\end{eqnarray*}	
This proves 
$
\dim {\rm Ad}L\, X_{\gs l}\geq \dim\g l - r
$	
and hence  $X_{\gs l}$ is regular in $\g l$.
\end{proof}

\bl
\label{lemxohxoh}
Let $\g g$ be a complex semisimple Lie algebra,
$\g h$ a complex  Lie subalgebra,
 and $X\in \g h^\bot$.
Then there exists $\g h'\in \ol{{\rm Ad}G\,\g h}$ such that 
$X\in {\g h'}^\bot$ and $[X,\g h']\subset \g h'$.
\el

We recall that $G$ is a connected	complex Lie group with Lie algebra $\g g$.
Such a Lie group has a unique structure of complex algebraic Lie group.

\begin{proof}[Proof of Lemma \ref{lemxohxoh}]
Let $A\subset G$ be the Zariski closure of the one-parameter subgroup
$\{e^{tX}\mid t\in \m C\}$. This group $A$ is abelian.

Note that, for all $a$ in $A$, the Lie subalgebra ${\rm Ad}a\,\g h$ is orthogonal to $X$.
Therefore, all Lie subalgebra $\g h'$ in the orbit closure $\ol{{\rm Ad}A\,\g h}$
are ortho\-gonal to $X$.
This orbit closure 	$\ol{{\rm Ad}A\,\g h}$ is a $A$-invariant subvariety 
of the pro\-jec\-tive algebraic variety $\mc L$. 
By Borel fixed point theorem
\cite[Theorem 10.6]{Bo69}, the solvable group $A$ has a fixed point in this 
subvariety. This means that there exists $\g h'$ in $\ol{{\rm Ad}A\,\g h}$
such that ${\rm Ad}A\,\g h'=\g h'$. In particular,  $[X,\g h']\subset \g h'$.
\end{proof}

\begin{proof}[Proof of the implication $\Longleftarrow$ in Proposition \ref{proslaorb}]
We argue by induction on the dimension of $\g g$. 
We assume that $\g h^\bot$ contains a regular element $X$,
and we want to prove $Sla(\g g,\g h)$. 
By Corollary \ref{corslaope}  and Lemma \ref{lemxohxoh}, we can also assume 
that $X$ normalizes $\g h$, i.e. that $[X,\g h]\subset \g h$.
In particular, the sum  $\wt{\g h}:=\m C X\oplus \g h$ is a Lie subalgebra of $\g g$.
By Lemma \ref{lemhhh} (i), we may and do assume that $\g h=[\g h,\g h]$.
Let $\g q$ be a parabolic subalgebra of $\g g$ of minimal dimension containing $\wt{\g h}$,
and $\g u$ the unipotent
radical of $\g q$.  
By minimality of $\g q$, the image of $\wt{\g h}$ in $\g q/\g u$
is reductive. Therefore 
we can write 
$\g h=\g s\oplus\g v$
where $\g s$ is a semisimple Lie subalgebra and 
$\g v:=\g h\cap \g u$ is the unipotent
radical of $\g h$. 
We can then write $\g q=\g l\oplus\g u$
where $\g l$ is a reductive Lie subalgebra containing $\g s$.
We sum up this discussion by the inclusions:
$$
\g h= \g s\oplus \g v\;\subset\; 
\g q=\g l\oplus\g u \;\subset\; \g g\, .
$$

Since $X$ is in  
$\widetilde{\mathfrak h} \subset \mathfrak q$, 
we can decompose $X$ as  $X=X_{\gs l}+X_{\gs u}$  with $X_{\gs l}\in \g l$ 
and $X_{\gs u}\in \g u$. By Lemma \ref{lemxxlpxu}, 
the element $X_\gs l$ is regular in $\g l$. 
Since $\g u$ is the orthogonal of $\g q$ with respect to the Killing form $K$, 
one has 
$$
K(X_{\gs l},\g s)=K(X_{\gs l}+X_{\gs u},\g s\oplus\g v)=K(X,\g h)=0.
$$
This proves that $X_{\gs l}$ is orthogonal to $\g s$.

We now claim that $\g q\neq \g g$. 
Indeed, if $\g q=\g g$, one has the equalities 
$\wt{\g h}=\g h=\g s$, and this Lie algebra is semisimple
 by the assumption that $\g h=[\g h, \g h]$.
Therefore the Killing form restricted to $\g h$ is nondegenerate.
This contradicts the assumption $X\in \g h^\bot$. 

Therefore one has $\g q\neq \g g$.
The normalizer $L:= N_G(\g l)$ of $\g l$ in $G$
has Lie algebra $\g l$. 
We have seen that the intersection $\g s^\bot\cap \g l_{reg}$ is non-empty. Therefore,
by induction hypothesis, the orbit closure $\ol{{\rm Ad}L\,\g s}$
contains a solvable Lie algebra, and
the orbit closure $\ol{{\rm Ad}L\,\g h}$ also
contains a solvable Lie algebra. This proves $Sla(\g g, \g h)$.
\end{proof}

\subsection{Rho and Sla}
\label{secrhosla}
\bq
In this section we will prove the following implication 
which is still purely algebraic. 
The proof of the converse will be much more delicate.
\eq
We will in fact prove a stronger statement

\bp
\label{prorhosla}
Let $\g g$ be a complex semisimple Lie algebra and $\g h\subset \g g$
be a complex Lie subalgebra. Then, one has the implication
\begin{eqnarray}
Rho(\g g,\g h)
&\Longrightarrow&
Sla(\g g, \g h)\, .
\end{eqnarray}
More precisely, if $\g h$ satisfies $Rho(\g g,\g h)$, then 
every Lie algebra $\g h'$ in $\ol{{\rm Ad}G\g h}$ 
satisfies $Sla(\g g, \g h)$.
\ep

It will be useful to introduce the following two 
$G$-invariant subsets of $\mc L$.
\begin{eqnarray}
\mc L_{rho}&:=&
\{\g h\in \mc L\mid 
\rho_{\gs h}\leq \rho_{\gs g/\gs h}\},\\
\mc L_{clo}&:=&
\{\g h\in \mc L\mid {\rm Ad} G\,\g h \;\;\mbox{\rm is closed in}\; \mc L\}.	
\end{eqnarray}

\begin{Rem}
We have the following nice characterisation 
of closed orbits in $\mc L$.
\begin{eqnarray}
\mbox{ }\hspace{-3em}\g h\in \mc L_{clo}&\Longleftrightarrow&
\mbox{the normalizer $N_\gs g(\g h)$ 
is a parabolic subalgebra of $\g g$}\\
&\Longleftrightarrow&	
\mbox{$\g h$ is normalized by a Borel 
subalgebra of $\g g$}
\end{eqnarray}
\end{Rem}

\begin{proof}[Proof of Proposition \ref{prorhosla}]
This follows from Lemma \ref{lemrhosla} below
and from the fact that the orbit closure always contains a closed $G$-orbit.
\end{proof}

\bl
\label{lemrhosla}
Let $\g g$ be a complex semisimple Lie algebra. Then, \\
$(i)$ $\mc L_{rho}$ is closed in $\mc L$.\\
$(ii)$ Let $\g h\subset \g g$
be a complex Lie subalgebra with ${\rm Ad}G\,\g h$  closed.
Then, 
\begin{eqnarray*}
\mbox{$\g h$ is solvable} 
&\Longleftrightarrow&
Rho(\g g, \g h)\, .
\end{eqnarray*}
\el

\begin{proof}[Proof of Lemma \ref{lemrhosla}]
$(i)$ The map $(\g h,Y) \mapsto \rho_ {\gs h}(Y)$ is continuous
on the set $\{ (\g h,Y) \mid\g h\in \mc L\; ,\; Y\in \g h\}$.
Let $\g h_n\in \mc L_{rho}$ be a sequence 
that converges to a Lie algebra $\g h_\infty$. We want to prove that $\g h_\infty\in \mc L_{rho}$. Let $Y_\infty\in \g h_\infty$.
We can find a sequence $Y_n\in \g h_n$ converging to $Y_\infty$.
Therefore, one has 
$$
\rho_\gs g(Y_\infty)-2\,\rho_{\gs h_\infty}(Y_\infty)
=\lim\limits_{n\ra\infty}
\rho_\gs g(Y_n)-2\,\rho_{\gs h_n}(Y_n)
\geq 0\, .
$$
This proves that $\g h_\infty$ is in $\mc L_{rho}$. 

$(ii) \Longrightarrow $  
Since $\g h$ is solvable, it is included in a Borel Lie subalgebra $\g b$. Note that
$\g b$ satisfies the $\rho$-inequality, more precisely, 
one has the equality 
$\rho_\gs b(Y)=\rho_{\gs g/\gs b}(Y)$, for all $Y$ in $\g b$.
Therefore, $\g h$ also satisfies $Rho(\g g,\g h)$.

$(ii) \Longleftarrow $  
Let $\g h$ be a Lie  subalgebra with ${\rm Ad}G\,\g h$ closed
and which satisfies $Rho(\g g,\g h)$. 
We want to prove that $\g h$ is solvable. 
Replacing a few times $\g h$ 
by its derived subalgebra
$[\g h,\g h]$ if necessary, 
we may assume that $\g h=[\g h,\g h]$.
Let $\g q$ be the normalizer of $\g h$
and $\g u$ be the unipotent radical of $\g q$.
By assumption $\g q$ is a parabolic Lie subalgebra.
The projection of $\g h$ in the reductive Lie algebra 
$\g q/\g u$ is an ideal and hence is a semisimple Lie algebra.
Therefore we can write $\g h=\g s\oplus \g v$,
where $\g s$ is a semisimple Lie subalgebra 
and $\g v:=\g h\cap \g u$ is the unipotent radical of $\g h$.
We  then write $\g q=\g l\oplus \g u$ where $\g l$
is a reductive Lie subalgebra containing $\g s$.
Let $\g u^-$ be the  opposite unipotent subalgebra
which is opposite to $\g q$ and normalized by $\g l$
so that $\g g=\g u^-\oplus\g l\oplus \g u$.
Fix $Y$ in $\g s$.
Since $\g q$ normalizes $\g h$ one has
\begin{eqnarray}
\label{eqnrhohlu}
\rho_\gs h(Y)&=& \rho_\gs l(Y)+\rho_\gs u(Y)\, .
\end{eqnarray}
Since $\g u^-$ is dual to $\g u$ as an $\g l$-module, one has
\begin{eqnarray}
\label{eqnrhoglu}
\rho_\gs g(Y)&=& \rho_\gs l(Y)+2\,\rho_\gs u(Y)\, .
\end{eqnarray}
Combining \eqref{eqnrhohlu} and \eqref{eqnrhoglu}, and using
the $\rho$-inequality, one gets
\begin{eqnarray*}
\label{eqnrhoglu2}
\rho_\gs s(Y)\;\leq\;\rho_\gs l(Y)&=& 
2\,\rho_\gs h(Y)-\rho_\gs g(Y)\;\leq\; 0\, .
\end{eqnarray*}
Since this is true for all $Y$ in the semisimple Lie algebra $\g s$,
one must have $\g s=0$. This proves that $\g h$ is solvable.
\end{proof}
 
\subsection{Reductive homogeneous spaces}
\label{secredhom}
\bq
In this section we check Theorem \ref{thmequcon}
for $\g h$ reductive by relying on the previous papers of this series.
We  will prove:
\eq

\bp
\label{proequcon}
Let $\g g$ be a complex semisimple Lie algebra
 and $\g h\subset \g g$ a complex reductive Lie subalgebra.
The following conditions are equivalent~: 
\begin{equation*}
\label{eqntemrhoslatmuorb2}
Tem(\g g,\g h)
\Longleftrightarrow Rho(\g g,\g h)
\Longleftrightarrow Sla(\g g,\g h)
\Longleftrightarrow Tmu(\g g,\g h)
\Longleftrightarrow Orb(\g g,\g h).
\end{equation*}
\ep

\begin{Rem} Since $\g g$ is semisimple and $\g h$ is reductive,
one has a decomposition $\g g=\g h\oplus\g h^\bot$
with respect to the Killing form, and  the orthogonal complement $\g h^\bot$ is isomorphic to the quotient $\g g/\g h$ as an $\g h$-module.
\end{Rem}

The proof uses  the condition $Ags(\g g,\g h)$ that we introduced in
\cite{BeKoIII} and proven to be equivalent
to $Rho(\g g,\g h)$. It is defined by:
\begin{eqnarray*}
\label{eqnags}
Ags(\g g,\g h)& :&
\mbox{\rm the set $\{X\in \g h^\bot\mid \g z_{\gs h}(X)$ 
is abelian $\}$  is dense in $\g h^\bot$. }\;
\end{eqnarray*}

\begin{proof}[Proof of Proposition \ref{proequcon}]\mbox{ }\\
$\star$ The equivalence  
$Tem(\g g,\g h)\Longleftrightarrow Rho(\g g,\g h)$
is proven in \cite[Thm.~4.1]{BeKoI} for all real semisimple Lie algebra $\g g$ 
and all real reductive Lie subalgebra $\g h$.\\ 
$\star$ The equivalence  
$Sla(\g g,\g h)
\Longleftrightarrow Tmu(\g g,\g h)
\Longleftrightarrow Orb(\g g,\g h)
$
has been proven in the previous sections for all complex Lie subalgebra $\g h$.\\ 
$\star$ The equivalence  
$Rho(\g g,\g h)\Longleftrightarrow Ags(\g g,\g h)$
is proven in \cite[Thm.~1.6]{BeKoIII} for all complex semisimple Lie algebra $\g g$ 
and all complex reductive Lie subalgebra $\g h$.\\ 
$\star$ The  equivalence  
$Ags(\g g,\g h)\Longleftrightarrow Orb(\g g,\g h)$
is proven in  Proposition \ref{proagsorb} below.
\end{proof}

\bp
\label{proagsorb}
Let $\g g$ be a complex semisimple 
Lie algebra and $\g h\subset \g g$
be a complex reductive Lie subalgebra. Then, one has the equivalence
\begin{eqnarray}
Ags(\g g,\g h)
&\Longleftrightarrow&
Orb(\g g, \g h)\, .
\end{eqnarray}
\ep

We will need the following lemma which relates centralizer in $\g g$ and centralizer in $\g h$.  

\bl
\label{lemminorb}
Let $\g g$ be a real semisimple Lie algebra, 
 $\g h$ a real reductive Lie subalgebra, 
 and regard $\g h^{\bot} \subset \g g$
 via the Killing form as before. 
Let 
$$
\g h^\bot_{\rm min}:=\{ X\in \g h^\bot\mid
\dim \g z_{\gs g}(X)=r_{\gs g,\gs h}\}
\;\;{\rm where}\;\;
r_{\gs g,\gs h}:=
\min\limits_{X\in\gs h^\bot}\dim \g z_{\gs g}(X)
$$
Then, for every $X_0$ in $\g h^\bot_{\rm min}$, 
one has 
$
[\g z_{\gs g}(X_0),\g z_{\gs g}(X_0)]\subset 
\g z_{\gs h}(X_0)\, .
$
\el

Note that Lemma \ref{lemminorb} applied to $\g h=\{0\}$ implies
 that $\g z_{\g g}(X_0)$ is abelian
 if $X_0 \in \g g_{\rm reg}$. Indeed, 
 when $\g h=\{0\}$, one has 
 $r_{\g g, \g h}=\operatorname{rank} \g g$
 and $\g h_{\rm min}^{\bot}=\g g_{\rm reg}$.

This lemma is a special case of the following  general lemma for coadjoint orbits of real Lie algebras which  is well-known when $\g h=\{0\}$.

\bl
\label{lemcoaorb}
Let $\g g$ be a real Lie algebra and $\g h\subset \g g$
be a real Lie subalgebra. 
Let $\g g^*$be the dual of $\g g$ and 
$\g h^\bot:=\{ f\in \g g^*\mid f(\g h)=\{0\}\}$.
We set 
$$
\g h^\bot_{\rm min}:=\{ f\in \g h^\bot\mid
\dim \g g_f=r_{\gs g,\gs h}\}
\;\;{\rm where}\;\;
r_{\gs g,\gs h}:=
\min\limits_{f\in\gs h^\bot}\dim \g g_f.  
$$
Then, for every $f_0$ in $\g h^\bot_{\rm min}$, 
one has 
$
[\g g_{f_{_0}},\g g_{f_{_0}}]\subset \g h_{f_{_0}}\, .
$
\el
Here $\g g_f:=\{ Y\in \g g\mid Yf=0\}$ denotes the stabilizer of $f$ in $\g g$ and 
$\g h_f:=\g g_f\cap \g h$ its stabilizer in $\g h$.

\begin{proof}[Proof of Lemma \ref{lemcoaorb}]
Fix $f_0\in \g h^\bot_{\rm min}$ and  two elements $Y_0$ and $Z_0$ in $\g g_{f_{_0}}$. We want to prove that $[Y_0,Z_0]\in \g h$.
We write 
$$
\g g= \g g_{f_{_0}}\oplus \g m
$$
where $\g m$ is a complementary vector subspace.

For all $f\in \g h^\bot$, for $t\in \m R$ small enough
the element $f_t:=f_0+tf$  is also in the open set $\g h^\bot_{\rm min}$. 
Choose a linear projection  $\pi_0 \colon \g g^* \to \g g f_0$.
By the local inversion theorem, the map 
\begin{eqnarray*}
\Phi \colon (Y_0+\g m)\times \m R
&\to &
\g g f_0\times \m R\\
(Y,t)
&\mapsto&
(\pi_0(Yf_t),t)
\end{eqnarray*}
is a local diffeomorphism near $(Y_0,0)$.
Let $t\mapsto Y_t$ be the differentiable curve 
near $0$ starting from $Y_0$ given by $\Ph(Y_t,t)=(0,t)$.  
Since for $t$ small the linear map 
$\pi_0 \colon \g gf_t\to\g gf_0$ is an isomorphism, it satisfies 
$$
Y_t\in Y_0+\g m
\;\;\;{\rm and}\;\;\;
Y_tf_{_t}=0\, .
$$	
For the same reason, there exists a differentiable curve $t\mapsto Z_t$
near $0$ starting from $Z_0$ such that 
$$
Z_t\in Z_0+\g m
\;\;\;{\rm and}\;\;\;
Z_tf_{_t}=0\, .
$$	
They satisfy the equality $f_t([Y_t,Z_t])=0$
whose derivative at $t=0$ gives 
$$
f([Y_0,Z_0])+f_0([Y'_0,Z_0])+f_0([Y_0,Z'_0])=0
$$
Since both $Y_0$ and $Z_0$ stabilize $f_0$ the last two terms are zero.
One deduces 
$$
f([Y_0,Z_0])=0\;\;
\mbox{\rm for all $f$ in $\g h^\bot$.}
$$
This proves that $[Y_0,Z_0]$ is in $\g h$ as required.
\end{proof}

The following lemma will also be useful.

\bl
\label{lemsshden}
Let $\g g$ be a complex semisimple Lie algebra and $\g h\subset \g g$
be a complex reductive Lie subalgebra. 
Then the set 
$$
\g h^\bot_{\rm ss}:=\{ X\in \g h^\bot\mid
\;\mbox{\rm $X$ is semisimple}\}.
$$
is Zariski dense in $\g h^\bot$.
\el

\begin{proof}[Proof of Lemma \ref{lemsshden}]
There exists a compact real form $\g g_{\m R}$ of $\g g$
such that $\g h$ is defined over $\m R$.
Since $\g g_{\m R}= \g h_{\m R} \oplus \g h_{\m R}^\bot$,
the vector space $\g h_{\m R}^\bot$ 
is Zariski dense in $\g h^\bot$. 
Since all elements of $\g g_{\m R}$ are semisimple, this proves our claim.
\end{proof}

\begin{proof}[Proof of Proposition \ref{proagsorb}]
$\Longleftarrow$ Since
the Zariski open set $\g g_{\rm reg}$ meets 
the or\-tho\-gonal $\g h^\bot$ for the Killing form, 
the intersection $\g h^\bot_{\rm reg}$ is dense in $\g h^\bot$.
By Lemma \ref{lemminorb} applied with the zero subalgebra, 
every $X_0$ in $\g g_{\rm reg}$ has an abelian centralizer in $\g g$.
In particular, every $X_0$ in $\g g_{\rm reg}$ has an abelian centralizer in $\g h$.
This proves $Ags(\g g,\g h)$.

$\Longrightarrow$ Let 
$r':= \min\{\dim \g z_{\gs h}(X)\mid X\in \g h^\bot\}.$
The set 
$$
\g h^\bot_{\rm gen}:=\{ X\in \g h^\bot_{\rm min}\mid
\dim \g z_{\gs h}(X)=r'\}
$$
is nonempty and Zariski open  in $\g h^\bot$.
By assumption the set 
$$
\g h^\bot_{\rm abe}:=\{ X\in \g h^\bot_{\rm gen}\mid
\g z_{\gs h}(X)\;
\mbox{is abelian}\}
$$
is dense in $\g h^\bot_{\rm gen}$.
Since it is also closed in $\g h^\bot_{\rm gen}$,
one has $\g h^\bot_{\rm abe}=\g h^\bot_{\rm gen}$.
Therefore by Lemma \ref{lemsshden} the set
$\g h^\bot_{\rm abe}$
 contains a semisimple element $X_0$.
The centralizer $\g z_{\gs g}(X_0)$ is then a reductive Lie algebra.
By Lemma \ref{lemminorb}, 
the  Lie algebra
$
[\g z_{\gs g}(X_0),\g z_{\gs g}(X_0)]$
is included in 
$\g z_{\gs h}(X_0)$ which is an abelian Lie algebra.
Therefore the Lie algebra  $\g z_{\gs g}(X_0)$ itself is abelian.
Since $X_0$ is semisimple, this centralizer is a Cartan subalgebra 
and $X_0$ is regular in $\g g$.
This proves $Orb(\g g, \g h)$.
\end{proof}


\section{Real algebraic homogeneous spaces}
\label{secreaalg}

The proof of the last remaining implication 
\eqref{eqnslarho} will last up to the end of this paper.
Because of the induction method which involves parabolic subgroups,
we need to extend the temperedness criterion
of \cite{BeKoII} to non-semisimple algebraic groups $G$. 
This extension will be valid for all real algebraic group.

\subsection{Notations}
\label{secnotrea}

Let $G$ be a 
real algebraic Lie group,
$H$ be an
algebraic Lie subgroup.
We write $G=LU$ and $H=SV$ where $S\subset L$ are reductive subgroups
and where $V$ and $U$ are the unipotent radical of $H$ and $G$.
Note that, in general one does not have the inclusion $V\subset U$. 
We denote by $\g g$, $\g h$, $\g l$, $\g u$, etc... the corresponding Lie algebras.

We consider the following conditions:
\begin{eqnarray*}
	\label{eqndeftemrhosla}
	Tem(\g g,\g h)& :&
	L^2(G/{{H}})\; \mbox{ is $L$-tempered}.\\
	Rho(\g g,\g h)&:&
	\rho_{\gs l}\leq 2\,\rho_{\gs g/\gs h}
	\;\mbox{as functions on $\g s$}.\\
	Sla(\g g,\g h)&:&  \ol{{\rm Ad} G\,\g h}
	\;\mbox{contains a solvable Lie algebra}.
\end{eqnarray*}
\begin{Rem}
By $L$-tempered, we mean tempered as a representation of $L$,
or, equivalently, tempered as a representation of the semisimple Lie group
$[L,L]$.	When $G$ is not semisimple this notion happens to be much more
useful than the temperedness as a representation of $G$.
\end{Rem}

\begin{Thm} 
\label{thmgh}
Let $G$ be a 
real algebraic Lie group,
$H$ be an 
algebraic Lie subgroup.
One has the equivalence,\\
\centerline{$Tem(\g g,\g h) \Longleftrightarrow Rho(\g g, \g h).$} 
\end{Thm}
\begin{Rem}
For real algebraic groups, the last condition $Sla(\g g,\g h) $
is not always equivalent to the first two, 
 but it is often the case.
For instance, we will see in Theorem \ref{thmtemrhosla}, 
that this is true for complex algebraic Lie groups.  
\end{Rem}
In the induction process, 
we will have to work with slightly 
more general representations than the regular representation $L^2(G/H)$.
Let $W$ be a finite-dimensional algebraic repre\-sentation of $H$.
We will have to deal with the ($L^2$-)induced representation 
${\rm Ind}_H^G(L^2(W))\simeq L^2(G\times_H W)$,
where $G\times_HW$ is the  
$G$-equivariant bundle
over $G/H$ with fiber $W$, 
see \cite[Section 2.1]{BeKoII} for more precise definition.
This is why we also introduce the following two conditions.
\begin{eqnarray*}
\label{eqndeftemrhow}
Tem(\g g,\g h,W)& :&
{\rm Ind}_H^G(L^2(W))\; \mbox{ is $L$-tempered}.\\
Rho(\g g,\g h,W)&:&
\rho_{\gs l}\leq 2\,\rho_{\gs g/\gs h}+2\,\rho_W
\;\mbox{as a functions on $\g s$}.
\end{eqnarray*}

The following theorem is a generalization of our Theorem 3.6 
in \cite{BeKoII} where we assumed that $G$ is semisimple.

\begin{Thm} 
\label{thmghw}
Let $G$ be a 
real algebraic Lie group,
$H$ be an 
algebraic Lie subgroup
and $W$ a finite-dimensional algebraic represen\-tation of $H$. 
One has the equivalence,\\
\centerline{$Tem(\g g, \g h,W) \Longleftrightarrow Rho(\g g, \g h,W).$} 
\end{Thm}

We have assumed here that $G$ and $H$ are algebraic 
only to avoid uninteres\-ting technicalities.
It is not difficult to get rid of this assumption.

\begin{proof}[Proof of Theorem \ref{thmgh}]
It is a special case of Theorem \ref{thmghw} 
with $W=0$.
\end{proof}
The proof of Theorem \ref{thmghw} follows the same line 
as in \cite[Theorem 3.6]{BeKoII}.\\ 
In this Chapter \ref{secreaalg}
we will prove the direct implication $\Longrightarrow$.\\
In the next Chapter \ref{secdomuni},
we will prove the converse implication $\Longleftarrow$.

\subsection{The Herz majoration principle}
\label{sechermaj}

\bq
We first recall a few lemmas on tempered representations
and on induced representations.
\eq

The first lemma is a variation on Herz majoration principle. 

\begin{Lem} 
\label{lemindtem}
Let $G$ be a real algebraic Lie group, 
$L$ be a reductive algebraic Lie subgroup of $G$
and $H$ be a closed subgroup of $G$.
If the regular representa\-tion in $L^2(G/H)$ is $L$-tempered then the induced representation 
$\Pi=\Ind_H^G(\pi)$ is also $L$-tempered
for any  unitary representation $\pi$ of $H$.
\end{Lem}

\begin{proof}
See for instance \cite[Lemma 3.2]{BeKoII}.
\end{proof}

The second lemma will 
prevent us to worry about connected components of $H$ and will
allow us to assume that $H=[H,H]$.

\begin{Lem} 
\label{lemghghh}
Let $G$ be a real algebraic Lie group, 
$L$ be a reductive algebraic subgroup of $G$
and $H'\subset H$ be two closed subgroup of $G$.	\\
$1)$ If $L^2(G/H)$ is $L$-tempered then $L^2(G/H')$ is $L$-tempered.\\
$2)$ The converse is true when $H'$ is normal in $H$ and $H/H'$ is amenable
(for instance finite,  compact, or abelian).
\end{Lem}

\begin{proof}
See for instance \cite[Proposition 3.1]{BeKoII}.
\end{proof}

The third lemma is good to keep in mind.

\begin{Lem}
\label{lemparsub2}
Let $Q=LU$ be a real algebraic Lie group which is
a semidirect product of a reductive subgroup $L$ 
and its unipotent radical $U$. 
Let $\pi_0$ be a unitary representation
of $Q$ which is $L$-tempered and trivial on $U$.
Then the representation $\pi_0$ is also $Q$-tempered.
\end{Lem}

\begin{proof}
See for instance \cite[Lemma 4.3]{BeKoII}.
\end{proof}

This lemma is useful for a parabolic subgroup $Q$
of a semisimple Lie group $G$. In this case the induced representation
${\rm Ind}_Q^G(\pi_0)$ is also $G$-tempered.

\subsection{Decay of matrix coefficients}
\label{sectemrep2}

\bq
We now recall the control of the matrix coefficients of tempe\-red 
representations of a reductive Lie group.
\eq

In the sequel, it will be more comfortable to deal with a reductive group $L$ than just with a semisimple group even though, 
in the temperedness condition, the center $Z_L$ of $L$ 
plays no role.

So, let $L$ be a real reductive
algebraic Lie group.
We fix a maximal compact subgroup $K$ of $L$ and denote by $\Xi$ 
the Harish-Chandra spherical function on $L$.
By definition, $\Xi$ is  the matrix coefficient of a normalized $K$-invariant vector $v_0$
of the spherical unitary principal representation 
$\pi_0={\Ind}_P^L({\bf 1}_P)$ 
where $P$ is a minimal parabolic subgroup of $L$.
That is 
\begin{equation}
\label{eqnxighar}
\Xi(\ell)=\langle\pi_0(\ell)v_0,v_0\rangle
\; ,\;\; 
\mbox{\rm for all $\ell$ in $L$}. 
\end{equation}
Since $P$ is amenable, the representation $\pi_0$ is $L$-tempered.

\begin{Prop} 
[Cowling, Haagerup and Howe {\cite{CHH}}]
\label{protemal2}
Let $L$ be a real algebraic reductive Lie group
 and $\pi$ be a unitary representation  of $L$. The following are equivalent:\\
$(i)$ the representation $\pi$ is tempered,\\
$(ii)$ for every $K$-finite vector $v$ in ${\mathcal H}_\pi $, for every $\ell$ in $L$,
one has 
$$
|\langle  \pi(\ell)v,v\rangle |\leq \Xi(\ell)\, 
\| v\|^2\dim \langle  K v\rangle .
$$
\end{Prop}

See \cite[Thms.~1, 2 and Cor.]{CHH}.
See also  \cite{HoTa}, \cite{Nev98} for other applications
of Proposition \ref{protemal2}.

For the regular representation in an $L$-space, this proposition becomes:

\begin{Cor}
\label{coralpcol}
Let $L$ be a real algebraic reductive Lie group
and $X$ be a locally compact space endowed with a continuous action of $L$ 
preser\-ving a Radon measure ${\rm vol}$.
The regular representation of $L$ in $L^2(X)$  is $L$-tempered if and only if,
for any  $K$-invariant compact subset $C$ of $X$, one has
\begin{equation}
\label{eqnvolgcc}
{\rm vol} (\ell\, C\cap C)\leq {\rm vol}( C)\;\Xi(\ell) 
\; ,\;\;\mbox{\rm for all $\ell$ in $L$}.
\end{equation}
\end{Cor}

Recall that the notation  $\ell\, C$ denotes the set  
$\ell\, C: =\{\ell x\in X: x\in C\}$.

\subsection{The function $\rho_V$}
\label{secfunrov}

\bq
We now explain, following \cite[Section 2.3]{BeKoII} 
how to deal with the functions $\rho_V$ 
occurring in the temperedness criterion. 
\eq

Let $H$ be a real algebraic Lie group, 
$\g h$ its Lie algebra and $V$ be a real algebraic
finite-dimensional representation of $H$.
For all element $Y$ in $\g h$,
we consider the eigenvalues of $Y$ in $V$ and
we denote by $V_+$ and $V_-$ 
the largest vector subspaces of $V$ on which 
the real part of all the eigenvalues of  $Y$
are respectively positive and negative, and we set
\begin{eqnarray*}
	\label{eqnrhopmv}
	\rho_V(Y)
	&: =& \tfrac12\,{\rm {\operatorname{Tr}}}(Y|_{V_+})-
	\tfrac12\,{\rm {\operatorname{Tr}}}(Y|_{V_-}).
\end{eqnarray*}
Let $\g a=\g a_{\gs h}$ be a maximal split abelian Lie subalgebra of $\g h$ 
{\it{i.e.}} the Lie subalgebra of a 
maximal split torus $A$ of $H$.
The function $\rho_V$ on $\g h$ is completely 
determined by its restriction to $\g a$.   
Let $P_V$ be the set of weights of $\g a$ in $V$ 
and, for all $\al$ in $P_V$, let $m_\al:=\dim V_{\al}$ 
be the dimension of the corresponding 
weight space. Then one has the equality
\begin{equation}\label{eqnrhovys}
\rho_V(Y) = \tfrac12\sum_{\al\in P_V}m_\al|\al(Y)|
\;\;\;\mbox{for all $Y$ in $\g a$.} 
\end{equation}

For example, when $\g h$ is semisimple and $V = \g h$
via the adjoint action,
our function $\rho_{\gs h}$ is equal on each  positive Weyl chamber $\g a_+$
of $\g a$ to the sum of the corresponding positive roots {\it{i.e.}} to twice the usual ``$\rho$'' 
linear form.

The functions $\rho_V$ occurs in the volume estimate 
of Corollary \ref{coralpcol} through the following  Lemma.

\begin{Lem}
\label{lemrovvol}
Let $V=\m R^d$. Let 
$\g a$ be an abelian split Lie subalgebra of ${\rm End}(V)$ and $C$ be a compact neighborhood of $0$ in $V$.
Then there exist constants $m_{_C } >0$ , $M_{_C }  > 0$ such that
$$
m_{_C } e^{-\rho_V(Y)}
\le e^{-{\rm Tr}(Y)/2}\,{\rm vol}(e^Y C \cap C)
\le M_{_C }  e^{-\rho_V(Y)}
\;\;\mbox{\rm for all $Y \in \g a$. }
$$
\end{Lem}

\begin{proof}
This is \cite[Lemma 2.8]{BeKoII}.
\end{proof}

\subsection{The direct implication}
\label{secdirimp}

\bq
We  first prove the direct implication in Theorem \ref{thmghw}
which is~:
\eq

\begin{Prop}
\label{proghw1}
Let $G$ be a real algebraic Lie group,
$H$ an algebraic Lie subgroup of $G$ and 
$W$ an algebraic representation of $H$. Let
$L$ be a maximal reductive subgroup of $G$
containing a maximal reductive subgroup $S$ of $H$.\\
If $\Pi:=\Ind_H^G(L^2(W))$ is $L$-tempered then one has 
$\rho_{\gs l}  \le 2\,\rho_{\gs g/\gs h} +2\,\rho_W$ on $\g s$.
\end{Prop}

\begin{proof}
This representation $\Pi$ is also the regular representation of the $G$-space
$X:=G\times_HW$. 
Let $A$ be a maximal split torus of $S$ and $\g a$ be the Lie algebra of $A$.
We choose an $A$-invariant decomposition 
$\g g=\g h\oplus \g m$ and small closed balls $B_0\subset \g m$ and
$B_W\subset W$ centered at $0$. 
We can see $B_W$ as a subset of $X$ and the map
\begin{eqnarray}
\label{eqnbobvgv}
B_0\times B_W \longrightarrow G\times_H W, 
\qquad
(u,v)\mapsto \exp (u)v 
\nonumber
\end{eqnarray}
is a homeomorphism onto its image $C$. 
Since $\Pi$ is $L$-tempered one has a bound
as in   \eqref{eqnvolgcc}
\begin{equation}
\label{eqnpigxig}
\langle\Pi(\ell )1_C,1_C\rangle\leq M_C\; \Xi(\ell)
\;\;\;\mbox{\rm for all $\ell$ in $L$}.
\end{equation}
We will exploit this bound for elements $\ell=e^Y$ with $Y$ in $\g a$.
In our coordinate system \eqref{eqnbobvgv} we can choose 
the measure $\nu_X$ to coincide with the Lebesgue measure on 
$\g m\oplus W$. Taking into account the Radon--Nikodym derivative and the $A$-invariance of $\g m$,
one computes as in \cite[Section 3.3]{BeKoII},
\begin{eqnarray*}
\label{eqnpigvol}
\langle\Pi(e^Y)1_C,1_C\rangle
\!\!&\geq &\!\! 
e^{-{Tr}_{\gs m}(Y)/2-{Tr}_{W}(Y)/2}\;
{\rm vol}_{\gs m}(e^Y\!B_0\cap B_0)\;{\rm vol}_W(e^Y\!B_W\cap B_W),
\end{eqnarray*}
and therefore, using Lemma \ref{lemrovvol}, one deduces
\begin{eqnarray}
\label{eqnpigrho}
\langle\Pi(e^Y )1_C,1_C\rangle
&\geq &m_{_C}\,e^{-\rho_{\gs m}(Y) } e^{-\rho_W(Y)} 
\;\;\;\mbox{\rm for all $Y$ in $\g a$.} 
\end{eqnarray}
Combining  \eqref{eqnpigxig} and \eqref{eqnpigrho} 
with known bounds for the 
spherical function $\Xi$ as in \cite[Prop 7.15]{Kn01},
one gets, for suitable positive constants $d$, $M_0$,
\begin{equation*}
\label{eqnroxiro}
\frac{m_C}{M_C}e^{-\rho_{\gs m}(Y)  -\rho_W(Y) }\leq \Xi(e^Y)\leq M_0\;(1+\|Y\|)^{d}e^{-\rho_{\gs l}(Y)/2 }  
\;\;\;\mbox{\rm for all $Y$ in $\g a$.} 
\end{equation*}
Therefore one has $\rho_{\gs l}\leq 2\,\rho_{\gs m}+2\,\rho_W$
as required.
\end{proof}

\section{Proof of temperedness for real groups}
\label{secdomuni}

In this Chapter, we prove the converse implication
in Theorem \ref{thmghw} which is~:

\begin{Prop}
\label{proghw2}
Let $G$ be a real algebraic Lie group,
$H$ an algebraic Lie subgroup of $G$ and $W$ an algebraic representation of $H$. Let
$L$ be a maximal reductive subgroup of $G$
containing a maximal reductive subgroup $S$ of $H$.\\
If $\rho_{\gs l}  \le 2\,\rho_{\gs g/\gs h} +2\,\rho_{_W}$ on $\g s$, 
then $\Pi:=\Ind_H^G(L^2(W))$ is $L$-tempered.
\end{Prop}

Recall that, when $W=0$, one has $\Pi=L^2(G/H)$.

\subsection{Domination of $G$-spaces}
\label{secdomspa}

\bq
The proof relies on the notion of domination of a $G$-action
that we have introduced in \cite{BeKoII}
without giving it a name.
\eq

Here is the definition. 
Let $G$ be a locally compact group.
Let  $X$ and $X_0$ be two locally compact spaces 
endowed with a continuous action of $G$,
and with a  $G$-invariant class of measures
${\rm vol}_X$ and ${\rm vol}_{X_0}$.
Let $\pi$ and $\pi_0$ be the unitary regular representations of $G$ 
in the Hilbert spaces of square-integrable half-densities 
$L^2(X)$ and $L^2(X_0)$.

\bd
[Domination of a $G$-space]
\label{defdomrep} We say that $X$ is 
{\it{$G$-dominated by $X_0$}}
if for every compactly supported bounded 
half-density $v$
on $X$,
there exists a compactly supported bounded 
half-density $v_0$ 
on $X_0$ such that, for all $g$ in $G$, 
\begin{equation}
\label{eqnpgvpgv}
|\langle\pi(g)v,v\rangle|
\leq
\langle\pi_0(g)v_0,v_0\rangle .
\end{equation}
\ed

\begin{Rem}
When both measures ${\rm vol}_X$ and ${\rm vol}_{X_0}$ are $G$-invariant,
the bound \eqref{eqnpgvpgv} means that, for every compact set $C\subset X$, 
there exists a constant $\la>0$ and a compact set $C_0\subset X_0$
such that, for all $g$ in $G$, 
$$
{\rm vol}(g\,C\cap C)
\leq
\la \,{\rm vol}(g\,C_0\cap C_0)
$$
\end{Rem}

This definition is very much related to our temperedness question because of the following lemma.

\begin{Lem}
\label{lemtemdom}
Let $G$ be a real algebraic reductive Lie group and
$P$ be a minimal parabolic subgroup of $G$,
and let  $X$ be a $G$-space.
The regular representation 	of $G$ in $L^2(X)$ is $G$-tempered
if and only if $X$ is $G$-dominated by the flag variety $X_0=G/P$.
\end{Lem}

\begin{proof}
This lemma is a direct consequence of 
Corollary \ref{coralpcol}.
\end{proof}

The following proposition gives us a nice situation
where an action is dominating another one.

\begin{Prop}
\label{prounidom}
Let $F=SU$ be a real algebraic Lie group which is a semidirect product of a reductive subgroup $S$ and 
its unipotent radical $U$. 
Let $H=SV$ be an algebraic subgroup of $F$ containing
$S$  where $V=U\cap H$. 
Let $Z$ be the $F$-space $Z=F/H=U/V$.
Let $Z_0:=Z$ 
endowed with another $F$-action 
where the $S$-action is the same but the $U$-action is trivial.
	
Then $Z$ is $F$-dominated by $Z_0$.
\end{Prop}

\begin{proof}
This is \cite[Corollary 4.6]{BeKoII}.
\end{proof}

\subsection{Inducing a dominated action}

\bq 
The following proposition 
tells us that the induction of actions preserves the domination. 
\eq
\begin{Prop} 
\label{proinddom}
Let $G$ be a locally compact group, 
 and $F$ a closed subgroup of $G$. Let $Z$ and $Z_0$ 
be two locally compact $F$-spaces
with $G$-invariant class of measures.
Let $X:=G\times_F Z$ and $X_0:= G\times_F Z_0$
be the two induced $G$-spaces. 

If $Z$ is $F$-dominated by $Z_0$ then $X$ is $G$-dominated by $X_0$.
\end{Prop}

\begin{proof}[Proof of Proposition \ref{proinddom}]
The proof is an adaptation of  \cite[Proposition 4.9]{BeKoII}
where $G$ was  an algebraic semisimple group.
We assume to simplify that 
the measures on $Z$ and $Z_0$ 
are $G$-invariant. 
This avoids to complicate the formulas with square roots
of Radon-Nikodym derivative.
The projection 
$$
G\rightarrow X':=G/F
$$
is a $G$-equivariant principal bundle with structure group $F$. 
	We fix
	a Borel measurable trivialization of this principal bundle 
	\begin{equation}
	\label{eqngsigh2}
	G\simeq X'\times F
	\end{equation}
	which sends relatively compact subsets to relatively compact subsets.
	The action of $G$ by left multiplication through this trivialization can be read as 
	\begin{equation*}
	\label{eqngxhgx2}
	g\,(x',f)=(gx',\si_F(g,x')f)
	\;\;\;\;\mbox{\rm 
		for all $g\in G$, $x'\in X'$ and $f\in F$,}
	\end{equation*}
	where $\si_F\colon G\times X'\rightarrow F$ 
	is a Borel  measurable cocycle. 
	This trivialization \eqref{eqngsigh2}
	induces a trivialization of the associated bundles
	\begin{eqnarray*}
		\label{eqngsigh3}
		X= G\times _F Z    
		&\simeq&  X' \times Z\, , \\    
		X_0 = G \times_F Z_0  
		&\simeq &                 
		X' \times Z_0\, .
	\end{eqnarray*}
	We start with a compact set $C$ of $X$.
	Through the first trivialization, 
	this compact set is included in a product 
	of two compact sets $C'\subset X'$ and $D\subset Z$
	\begin{equation}
	\label{eqncsumxd}
	C\subset C'\times D\; .
	\end{equation}
	Since $Z$ is $F$-dominated by $Z_0$ there exists $\la>0$ 
	and a compact subset 
	$D_0\subset Z_0$ such that, for all $f$ in $F$ ,
	$$
	{\rm vol}_Z(f\,D\cap D)
	\leq
	\la \,{\rm vol}_{Z_0}(f\,D_0\cap D_0)
	$$
	We compute, 
	for $g$ in $G$,
	\begin{eqnarray*}
		{\rm vol}_X(g\, C\cap C)
		&\leq&
		\int_{gC'\cap C'}{\rm vol}_Z(\si_F(g,g^{-1} x')D\cap D)\rmd x'\\
		&\leq &
		\la \int_{gC'\cap C'}{\rm vol}_{Z_0}(\si_F(g,g^{-1} x')D_0\cap D_0)\rmd x'\\
		&\leq &
		\la \,{\rm vol}_{X_0}(g\, C_0\cap C_0),
	\end{eqnarray*} 
	where $\rmd x'$ is a $G$-invariant measure on $X'$
	and $C_0$ is a compact subset of $X_0\simeq X'\times Z_0$
	which contains $C'\times D_0$.
\end{proof}

\subsection{The converse implication}
\label{secsemgro}

\bq
We conclude the proof 
of the converse implication in Theorem \ref{thmghw},
by reducing it to the case where $G$ is {\it{reductive}}
which was proven in \cite[Theorem 3.6]{BeKoII}
\eq

We will need the following  lemma
on the structure of nilpotent homoge\-neous spaces.
See \cite[Lemma 4.7]{BeKoII}, for a similar statement.
We recall that a unipotent Lie group is an algebraic nilpotent Lie group
with no torus factor.

\bl
\label{lemuvw}
Let $U$ be a real unipotent Lie group, $V$ a unipotent subgroup and
$\g v\subset\g u$ their Lie algebra.\\
{\rm{(1)}} There exists 
a real vector subspace $\g m\subset \g u$ such that 
$\g u=\g m\oplus \g v$ and the exponential map induces
a polynomial bijection ${\rm exp}\colon \g m \overset \sim \to U/V$.\\
{\rm{(2)}} Moreover, if $\g v$ is invariant by a reductive subgroup 
$S\subset {\rm Aut}(\g u)$, one can choose $\g m$
to be $S$-invariant. 
\el

\begin{proof}[Proof of Lemma \ref{lemuvw}]
We proceed by induction on $\dim U$. Let $Z$ be the center of $U$ and $\g z$
its Lie algebra.

{\bf First case : $\g z\cap \g v\neq \{0\}$.} 
In this case we apply the induction assumption to the Lie algebra 
$\g u':=\g u/(\g z\cap\g v)$ and its Lie subalgebra
$\g v':=\g v/(\g z\cap \g v)$.
This gives us 
an $S$-invariant subspace $\g m'$ of $\g u'$ such that 
$\g u'= \g m'\oplus\g v'$ and ${\rm exp}\colon \g m'\to U'/V'\simeq U/V$
is a bijection.
We denote by $\pi\colon \g u\to \g u'$ the projection and 
choose $\g m$ to be any $S$-invariant vector subspace
of 
$\pi^{-1}\mathfrak m'$ such that $\g m\oplus (\g z\cap\g v)=\pi^{-1}\g m'$. 

{\bf Second case : $\g z\cap \g v= \{0\}$.} 
In this case we apply the induction assumption to the Lie algebra 
$\g u':=\g u/\g z$ and its subalgebra
$\g v':=(\g v\oplus\g z)/\g z$.
This gives us 
an $S$-invariant subspace $\g m'$ of $\g u'$ such that 
$\g u'= \g m'\oplus\g v'$ and 
${\rm exp}\colon \g m'\to U'/V'$
is a bijection.
We denote by $\pi\colon \g u \to \g u'$ the projection and 
choose $\g m:=\pi^{-1}\g m'$. 
The identifications
$\g m' \simeq\g m/\g z$  and
$U'/V'\simeq U/VZ$ prove that
the exponential map ${\rm exp}\colon \g m\to U/V$ is bijective.
\end{proof}

\begin{proof}[Proof of Proposition \ref{proghw2}] 
We distinguish two cases.

{\bf First case : $W= \{0\}$.} In this case, one has  $\Pi=L^2(G/H)$.
We denote by $U$ and $V$ the unipotent radical of $G$ and $H$,
so that we have the equalities $G=LU$ and $H=SV$.
We have the inclusion $S\subset L$, but the group $V$ might not be included in $U$.
We introduce the unipotent group  $V':=VU\cap L$
and the algebraic groups $F:=HU$ and $F':=F\cap L$ 
so that we have the equality $F'=SV'$ and the inclusions
$$
H=SV\subset F=F'U\subset  G=LU\, .
$$
Let 
$$
Z:=F/H 
$$ 
and let $Z_0$ be the $F$-space $Z$ 
endowed with the same $S$-action but with a trivial $VU$-action.
One can easily describe $Z_0$. Indeed, let
$\g u$, $\g v$,... be the Lie algebras 
of $U$, $V$,... 
By Lemma \ref{lemuvw}, $Z_0$ can be identified with the $S$-module 
$W':=\g u/(\g u\cap \g v),$ 
 as is seen from the following isomorphisms:
$$
F/H\simeq VU/U\simeq U/(U\cap V)\simeq \g u/(\g u\cap \g v)\, .
$$
According to Proposition \ref{prounidom},
the $F$-space $Z$ is dominated by $Z_0$.
We intro\-duce now the two induced $G$-spaces
$$
X:=G\times_F Z=G/H
\;\;\;\mbox{\rm and}\;\;\;
X_0:= G\times_F Z_0\, .
$$
According to Proposition \ref{proinddom}, 
the $G$-space $X$ is dominated by  $X_0$.
Hence
$$
\mbox{\rm
the $L$-space 
$X=G/H$ is dominated by the $L$-space 
$X_0=L\times_{F'}W'$}
$$
By assumption one has 
\begin{equation*}
\label{eqnparsub1}
\rho_{\gs l}\leq 2\,\rho_{\gs g/\gs h}. 
\end{equation*}
Since 
$
\rho_{\gs g/\gs h}= \rho_{\gs g/\gs f}+\rho_{\gs f/\gs h}
=\rho_{\gs l/\gs f'}+\rho_{\gs u/(\gs u\cap\gs v)}\, ,
$
this can be rewritten as 
\begin{equation*}
\label{eqnparsub2}
\rho_{\gs l}\leq 2\,\rho_{\gs l/\gs f'}+2\, \rho_{W'}\, . 
\end{equation*}
Since $L$ is reductive, we can apply  \cite[Theorem 3.6]{BeKoII}. This tells us that 
the representation $L^2(L\times_{F'}W')$ is $L$-tempered.

Therefore since the $L$-space $X$ is $L$-dominated by $X_0$
the represen\-tation of $L$ in $L^2(G/H)$ 
is $L$-tempered, as required.
\vs

{\bf Second case : $W\neq  \{0\}$.} In this case, one has $\Pi=L^2(G\times_H W)$.
For $w$ in $W$, we denote by $H_w$ the stabilizer of $w$ in $H$.
We write $H_w=S_wU_w$ with $S_w$ reductive and $U_w$ the unipotent radical. 
Since the action of $H$ on $W$ is algebraic, there exists a Borel measurable subset 
$T\subset W$ which meets each of 
these $H$-orbits in exactly one point.
We can assume that for each $w$ in $T$, one has $S_w\subset S$. 
Let $\mu$ be a probability measure on $W$ with positive density and $\nu$ be the probability measure on $T\simeq S\backslash W$ given as the image of $\mu$.
One has an integral decomposition 
of the regular representation
\begin{equation}
\label{eqnregint0}
L^2(G\times_H W) =\int_T^{\oplus}L^2(G/H_w)\rmd\nu(w).
\end{equation}
Since the direct integral of tempered represen\-tations is tempered,
we only need to prove that, for $\nu$-almost all $w$ in $T$,
\begin{equation}
\label{eqnghvtem0}
L^2(G/H_w) 
\;\;\mbox{is $L$-tempered.}
\end{equation}
We can choose $w$ in the Zariski open set where dim $H_w$ is minimal.
According to \cite[Lemma 3.9]{BeKoII}, for such a $w$,
\begin{equation}
\label{eqnacttri}
\mbox{the action of $H_w$ on $W/(\g h\, w)$ is trivial.}
\end{equation}
Our assumption implies that one has the inequality on $\g s_w$
$$
\rho_\g l \leq 2\,\rho_{\g g/\g h}+2\,\rho_W\, .
$$
Thanks to \eqref{eqnacttri}, this can be rewritten as
$$
\rho_\g l \leq 2\,\rho_{\g g/\g h}+2\,\rho_{\g h/\g h_w}
\; =\; 
2\,\rho_{\g g/\g h_w}\, .
$$
Then the first case tells us that for such $w$,
the representation of $L$ in $L^2(G/H_w)$ is tempered.
This proves \eqref{eqnghvtem0} as required. 
\end{proof}

\subsection{Using parabolic subgroups}
\label{secparsub}

\bq
The aim of this section is to explain how,
when dealing with a quotient
$G/H$ of real algebraic groups, one can, using parabolic subgroups, 
reduce to the case where 
the unipotent radical $V$ of $H$
is included in the unipotent radical $U$ of $G$.
This reduction method will be used in Chapter \ref{seccomalg}
for complex Lie groups.
\eq

Let $G$ be a real algebraic  Lie group
and $H$ a real
algebraic subgroup of $G$.
We write $G=LU$ and $H=SV$ where $U$ and $V$ are the unipotent radicals of $G$ and $H$, and where $S$ and $L$ are reductive algebraic subgroups.
We can manage so that $S\subset L$ but we cannot always assume that 
$V$ is included in $U$. For instance this is not possible when $G$ 
is reductive and $H$ is not. 
We fix a parabolic subgroup $G_0$ of $G$ 
that contains $H$ and which is minimal with this property.
We denote by $U_0\supset U$ the unipotent radical of $G_0$. 
\begin{Lem}
	\label{lemparsub}
	One has the inclusion $V\subset U_0$.
	Moreover, we can choose a reductive subgroup $L_0\subset G_0$ such that 
	$G_0=L_0U_0$ and $S\subset L_0$.
\end{Lem}

\begin{proof}
The group $V_0:=U_0\cap H$ is a unipotent normal subgroup of $H$. 
The quotient $S':=H/V_0$ is an algebraic 
subgroup of
the reductive group $G_0/U_0$ which is not contained in any proper
parabolic subgroup of $G_0/U_0$. 
Therefore, by \cite[Sec.~VIII.10]{Bou7a9}
this group $S'$ is reductive and  the group $V_0$ is the unipotent radical $V$ of $H$. This proves the inclusion $V\subset U_0$.
	
	Since maximal reductive subgroups $L_0$ of $G_0$ are $U_0$-conjugate,
	one can choose $L_0$ containing $S$.
\end{proof}

We introduce 
the $L_0$-module $W_0:=\g u_0/\g v$.
The following two lemmas will be useful 
in our induction process.

\begin{Prop}
\label{proredpar}
Keep this notation. 
The following are equivalent:\\
$(i)$ \;\;$L^2(G/H)$ is $L$-tempered;\\
$(ii)$ \; $\rho_\gs l\leq 2\,\rho_{\gs g/\gs h}$ as a function on $\g s$;\\
$(iii)$ $L^2(G_0/H)$ is $L_0$-tempered;\\
$(iv)$ \;$\rho_{\gs l_0}\leq 2\,\rho_{\gs g_0/\gs h}$ as a function on $\g s$;\\
$(v)$\;\; $L^2(L_0\times _SW_0)$  is $L_0$-tempered.  
\end{Prop}

\begin{proof}[Proof of Proposition \ref{proredpar}]
$(i)\Leftrightarrow (ii)$ and $(iii)\Leftrightarrow (iv)$.  
This is Theorem \ref{thmgh}.

$(ii)\Leftrightarrow (iv)$ Write $\g u_0=\g u'_0\oplus \g u$ 
where $\g u'_0:=\g u_0\cap\g l$.
The equivalence follows from the equalities
$\rho_\gs l=\rho_{\gs l_0}+2\, \rho_{\gs u'_0}$ and
$\rho_\gs g=\rho_{\gs g_0}+ \rho_{\gs u'_0}$ .

$(iv)\Leftrightarrow (v)$ 
This follows from Theorem \ref{thmghw} if one notices the equality
$\rho_{\gs g_0/\gs h}=\rho_{\gs l_0/\gs s}+\rho_{W_0}$.
\end{proof}

The following lemma will  also be useful in this reduction process.

\begin{Lem}
\label{lemredpar}
Keep this notation.  
The following are equivalent:
\\
{\rm{(i)}} the orbit closure $\ol{{\rm Ad G}\,\g h}$
contains a solvable Lie algebra;\\
{\rm{(ii)}} the orbit closure $\ol{{\rm Ad G_0}\,\g h}$
contains a solvable Lie algebra.
\end{Lem}

\begin{proof}[Proof of Lemma \ref{lemredpar}]
This follows from the compactness of $G/G_0$.
\end{proof}

\section{Complex algebraic homogeneous spaces}
\label{seccomalg}

The aim of this chapter is to prove the last remaining impli\-cation 
in Theorem \ref{thmequcon} which is the converse of Proposition 
\ref{prorhosla}. 
We keep the notation of the previous Chapters 
\ref{secreaalg} and \ref{secdomuni}. We assume  in this chapter that both $G$ and $H$ are complex algebraic Lie group, but do not assume $G$ to be semisimple.

\subsection{The equivalence for $G$ algebraic}
\label{secslarho}

We first state the  
extension of Theorem \ref{thmequcon},
which relates temperedness to the existence of 
solvable limit algebras for a general algebraic group $G$.
This extension will be useful because of the 
induction process in the proof.
We still use 
the notation in Section \ref{secnotrea}.

\begin{Thm} 
	\label{thmtemrhosla}
	Let $G$ be a complex algebraic Lie group and
	$H$ be a complex algebraic  subgroup.
	Then one has the equivalences,
	$$Tem(\g g, \g h) \Longleftrightarrow Rho(\g g, \g h)\Longleftrightarrow Sla(\g g, \g h).$$
\end{Thm}

\begin{proof}[Proof of Theorem \ref{thmtemrhosla}]
The first equivalence follows from Theorem \ref{thmgh}.
We split the proof of the second equivalence 
into  Propositions \ref{prorhosla2}
and \ref{proslarho2}.
\end{proof}

\begin{Cor}
	\label{corghghgh}
	Let $G$ be a complex algebraic Lie group,
	$H$ be a  complex algebraic  subgroup,
	and $\g h'\in \ol{{\rm Ad}G\,\g h}$.
	Then one has the equivalence,
	$$Sla(\g g, \g h) \;\;\Longleftrightarrow Sla(\g g, \g h').
	$$
\end{Cor}

This equivalence  says that if a Lie subalgebra 
admits one solvable limit, then all its limit
Lie algebras also admit a solvable limit.

\begin{proof}[Proof of Corollary \ref{corghghgh}]
More precisely it is a corollary of Propositions \ref{prorhosla2}
and \ref{proslarho2}. 
Indeed, if $\g h$ satisfies $Sla(\g g,\g h)$, then by Proposition 
\ref{proslarho2}, it satisfies $Rho(\g g,\g h)$. 
Then by Proposition \ref{prorhosla2}, all limit subalgebras 
$\g h'\in \ol{{\rm Ad}G\,\g h}$ also satisfy $Sla(\g g,\g h')$.
\end{proof}

\begin{Rem}
\label{remslaclo}
The set of Lie subalgebras
$\g h$ in $\g g$ satisfying $Sla(\g g, \g h)$ is closed. 
Indeed, this follows from the $Rho$-condition in Theorem \ref{thmtemrhosla}.
\end{Rem}

\subsection{Rho and Sla}
\label{secrhosla2}
\bq
We extend Proposition \ref{prorhosla} to general
algebraic groups $G$.
\eq

\bp
\label{prorhosla2}
Let $\g g$ be an algebraic complex Lie algebra and $\g h\subset \g g$
be a complex Lie subalgebra. Then, one has the implication
\begin{eqnarray*}
Rho(\g g,\g h)
	&\Longrightarrow&
	Sla(\g g, \g h)\, .
\end{eqnarray*}
More precisely, if $\g h$ satisfies $Rho(\g g,\g h)$, then 
every Lie algebra $\g h'$ in $\ol{{\rm Ad}G\g h}$ 
satisfies $Sla(\g g, \g h)$.
\ep

\begin{Rem} In Propositions \ref{prorhosla2}
and \ref{proslarho2},
the assumption that $\g g$ is algebraic, 
i.e. is the Lie algebra of a complex algebraic Lie group
can easily be removed. We will not need it.
\end{Rem}

\begin{proof}[Proof of Proposition \ref{prorhosla2}]
This follows from Lemma \ref{lemrhosla2} below
and from the fact that the orbit closure always contains a closed $G$-orbit.
\end{proof}

We denote again by $\mc L_{rho}$ the set of Lie subalgebras 
$\g h$ of $\g g$ that satisfy $Rho(\g g,\g h)$.

\bl
\label{lemrhosla2}
Let $\g g$ be an algebraic complex Lie algebra. Then, \\
$(i)$ $\mc L_{rho}$ is closed in $\mc L$.\\
$(ii)$ Let $\g h\subset \g g$
be a complex Lie subalgebra with ${\rm Ad}G\,\g h$ closed.
Then, 
\begin{eqnarray*}
\mbox{$\g h$ is solvable} 
&\Longleftrightarrow&
Rho(\g g, \g h)\, .
\end{eqnarray*}
\el

\begin{proof}[Proof of Lemma \ref{lemrhosla2}] This is a straightforward extension of Lemma \ref{lemrhosla}. 
We write $\g g=\g l\oplus \g u$ with $\g l$ reductive and $\g u$
the unipotent radical.

$(i)$ Same as for Lemma \ref{lemrhosla}.

$(ii) \Longrightarrow $ Same  as for Lemma \ref{lemrhosla}, but note that
for $\g h=\g b\oplus \g u$ with $ \g b$  a Borel subalgebra of $\g l$,
one has 
$\rho_\gs l=2\, \rho_{\gs l/\gs b} =2\,\rho_{\gs g/\gs h}$.

$(ii) \Longleftarrow $  
We may assume that $\g h=[\g h,\g h]$.
Let $\g q$ be the normalizer of $\g h$.
By assumption $\g q$ is a parabolic Lie subalgebra of $\g g$
and $\g h$ is an ideal of $\g q$.
Let $\g g_0$ be a parabolic subalgebra of $\g q$ containing $\g h$
and which is minimal with this property.
We can write $\g g_0=\g l_0\oplus \g u_0$ and $\g h=\g s\oplus \g v$,
where $\g l_0$
is a reductive Lie algebra, where $\g u_0$ is the unipotent radical of $\g g_0$,
where $\g s:=\g h\cap\g l_0$ is an ideal of $\g l_0$ and 
where $\g v:=\g h\cap \g u_0$.
By assumption one has $Rho(\g g,\g h)$.
Then, by the equivalence $(ii)\Leftrightarrow (iv)$ 
in Proposition \ref{proredpar} one also has
$Rho(\g g_0,\g h)$ i.e. 
$$
\rho_{\gs l_0}\leq 2\, \rho_{\gs g_0/\gs h}
\;\;\mbox{ as a function on $\g s$.}
$$
But since $\g h$ is an ideal in $\g g_0$, the right hand side is null and
this inequality can be rewritten as
$\rho_\gs s\leq 0$. This tells us that $\g s$ is abelian and $\g h$ is solvable.
\end{proof}

\subsection{Sla and Rho}
\label{secslarho2}

We are now able to prove the last remaining implication
\eqref{eqnslarho} by proving the following stronger 
Proposition \ref{proslarho2}  
which is the converse to Proposition \ref{prorhosla2}. 

\bp
\label{proslarho2}
Let $\g g$ be a complex algebraic Lie algebra and $\g h\subset \g g$
be a complex Lie subalgebra. Then, one has the implication
\begin{eqnarray*}
Sla(\g g,\g h)
&\Longrightarrow&
Rho(\g g, \g h)\, .
\end{eqnarray*}
\ep

\begin{proof}[Beginning of proof of Proposition \ref{proslarho2}]
The proof of Proposition \ref{proslarho2} 
will be by induction on the dimension of $\g g$,
reducing to the case where both $\g g$ and $\g h$ are semisimple
that we discussed in Proposition \ref{proequcon}.
Using Lemma \ref{lemghghh} and Theorem \ref{thmgh}, 
we can assume that $\g h=[\g h,\g h]$.  
In Proposition \ref{proredpar} and Lemma \ref{lemredpar},
we have introduced an intermediate algebraic complex Lie algebra
$\g h\subset \g g_0\subset \g g$
such that the unipotent radical $\g v$ of $\g h$
is included in the unipotent radical $\g u_0$ of $\g g_0$,
and for which we have the equivalences~:
\begin{eqnarray*}
Rho(\g g,\g h)
\Longleftrightarrow
Rho(\g g_0,\g h)
&{\rm and}&
Sla(\g g,\g h)
\Longleftrightarrow
Sla(\g g_0,\g h).
\end{eqnarray*}
The proof will go on for two more sections.
\end{proof}

\subsection{Pushing down the Sla condition}
\label{secconsla}

\bq
We sum up
the previous notation.
\eq

\centerline{\bf Notation}
\noindent
Let $G_0=L_0U_0$ be an algebraic complex Lie group,\\
where $L_0$ is reductive and $U_0$ is the unipotent radical of $G_0$.\\
Let $H=SV$ be a connected algebraic complex Lie subgroup,\\
where $S$ is reductive and $V$ is the unipotent radical of $H$.\\
Assume that $S\subset L_0$ and $V\subset U_0$, 
and let $W_0:=U_0/V$.\\
For $w$ in $W_0$, we denote by $S_w$ the stabilizer of $w$ in $S$.\\
Let $\g g_0$, $\g h$,..., $\g s_w$ be the corresponding Lie algebras.

\bl
\label{lemdownsla}
Keep this notation. If $\g h$ satisfies $Sla(\g g_0,\g h)$,
then there exists a non-empty Zariski open set $W_0'\subset W_0$ 
such that for all $w$ in $W_0'$, $\g s_w$ satisfies $Sla(\g l_0,\g s_w)$  
\el

\begin{proof}[Proof of Lemma \ref{lemdownsla}]
By Lemma \ref{lemuvw}, 
there exists an $S$-invariant 
vector subspace $\g m\subset \g u_0$ such that 
$\g u_0=\g m\oplus \g v$ and the map ${\rm exp}\colon \g m \to W_0=U_0/V$
is a bijection.

By assumption, there exists a sequence $g_n\in G_0$ such that 
the limit
\begin{eqnarray}
\g h_\infty
&:=&
\lim_{n\ra\infty}{\rm Ad} g_n\, \g h
\end{eqnarray}
exists and is a solvable Lie subalgebra of $\g g_0$.

Since $V$ normalizes $\g h$, we can assume that
\begin{eqnarray}
g_n=\ell_n e^{X_n}
&{\rm with}&
\ell_n\in L_0\; {\rm and}\; X_n\in \g m.
\end{eqnarray}
We denote by $w_n\in W_0$ the image  $w_n:={\rm exp }(X_n)$.
The stabilizer $\g s_{w_n}$ of $w_n$ in $\g s$ is also the centralizer
of $X_n$ in $\g s$.
Therefore, one has the equality  
\begin{eqnarray}
{\rm Ad} e^{X_n}\, \g s_{w_n}
&=&
\g s_{w_n}.
\end{eqnarray}
Therefore, after extraction the limit
$
\g s_\infty
:=
\lim\limits_{n\ra\infty}{\rm Ad} \ell_n\, \g s_{w_n}
$
exists and is a Lie subalgebra of $\g h_\infty$.
In particular, this limit $\g s_\infty$ is solvable.
Therefore there exists a maximal unipotent Lie algebra $\g n_0$
of $\g l_0$ such that 
$$
\g s_\infty\cap \g n_0=\{ 0\},
$$
and, for $n$ large, one also has 
$
{\rm Ad}\ell_n\g s_{w_n}\cap \g n_0=\{ 0\}.
$
We have found at least one point $w_0$ in $W_0$
whose stabilizer $\g s_{w_0}$ is transversal to a maximal 
unipotent subalgebra $\g n$ of $\g l_0$. 
For such a subalgebra $\g n$ the set 
$$
W_0':=\{w\in W_0\mid \g s_w\cap \g n=\{0\}\}
$$
is a non-empty Zariski open subset of $W_0$. 
 
By the equivalence of  $Sla$ and $Tmu$
proven in Proposition \ref{proslatmu},
and since $\g l_0$ is reductive, for all  $w$ in $W_0'$,
the stabilizer $\g s_w$ satisfies $Sla(\g l_0,\g s_w)$.
\end{proof}

\subsection{Pushing up the Rho condition}
\label{secindrho}

\bq
We now explain how a disintegration argument
allows us to push  the $Rho$-condition 
from 
$(\g l_0,\g s_w)$ up to $(\g g_0,\g h)$.
It is very surprising that we need this analytic argument 
to relate these two algebraic conditions.
\eq

\begin{proof}[End of proof of Proposition \ref{proslarho2}] 
We keep the notation of Sections \ref{secparsub} and \ref{secconsla}, 
and we go on the proof by induction on the dimension of $G$.
\vs

{\bf First case : $L_0\neq  G$. } 
We want to prove the condition $Rho(\g g,\g h)$.
We first check that the regular representation of $L_0$
in $L^2(L_0\times_S W_0)$ is tempered.
We argue as in the second case of Section \ref{secsemgro}.
As in \eqref{eqnregint0}, we write the representation 
$L^2(L_0\times_S W_0)$ as an integral of
$L^2(L_0/S_w)$
so that we only need to prove that, for Lebesgue almost all $w$ in $W_0$,
the representation
\begin{equation}
\label{eqnghvtem}
L^2(L_0/S_w) 
\;\;\mbox{is $L_0$-tempered.}
\end{equation}
Note that the non-empty Zariski open set $W'_0$
introduced in Lemma \ref{lemdownsla}  has full Lebesgue measure.
We have seen in Lemma \ref{lemdownsla} that 
$$
\mbox{
$\g s_w$ satisfies $Sla(\g l_0,\g s_w)$,\;\; for all $w$ in  $W'_0$.}
$$
Since $\dim L_0 <\dim G$, our induction assumption
implies that 
$$
\mbox{
$\g s_w$ satisfies $Rho(\g l_0,\g s_w)$,\;\; for all $w$ in  $W'_0$.}
$$
And therefore by Theorem \ref{thmgh}, 
$$
\mbox{
$\g s_w$ satisfies $Tem(\g l_0,\g s_w)$,\;\; for all $w$ in  $W'_0$.}
$$
This proves \eqref{eqnghvtem} and the representation  of $L_0$ in
$L^2(L_0\times_S W_0)$ is tempered.

Finally, using Proposition \ref{proredpar}, 
one deduces that $L^2(G/H)$ is $L_0$-tempered, 
or equivalently $\g h$ satisfies $Rho(\g g,\g h)$.
\vs

{\bf Second case : $L_0= G$. }
In this case both $G$ and $H$ must be 
reductive. As we have seen in Lemma \ref{lemghghh},
we can assume that  $\g h=[\g h,\g h]$.
We can also assume that $\g g=[\g g,\g g]$.
Therefore one is reduced to the case where both $\g g$ and $\g h$
are semisimple which was settled in Proposition \ref{proequcon}.
This ends the proof of Proposition \ref{proslarho2}.
\end{proof}

This also ends simultaneously the proofs of Theorems \ref{thmtemorb}, \ref{thmequcon} and \ref{thmtemrhosla}.   

\subsection{Comments and perspectives}
\label{seccomper}

\bq
We conclude by a few remaining questions
\eq

\subsubsection{Openness of the Sla condition}
\begin{question}
\label{conslaope}
Let $\g g$ be a complex Lie algebra.
Is the set of Lie subalgebras $\g h$ 
satisfying $Sla(\g g, \g h)$  an open set? 
\end{question}

We have seen that this set is closed in Remark \ref{remslaclo}
and we have seen that this set is open when $\g g$ is semisimple 
in Corollary \ref{corequcon2}.

\subsubsection{Regular finite-dimensional representation}
Let $\g g$ be a complex semisimple Lie algebra and $\g h$ be a complex Lie subalgebra. We denote  by  $Irr(\g g)_{reg}$ 
the set of finite-dimensional 
irreducible represen\-tations $V$ of $\g g$ whose highest weight
is regular.  
We now consider the  condition 
\begin{eqnarray*}
\label{eqnregrep}	
Rep(\g g,\g h)&:&
\mbox{there exists\; $V\in Irr(\g g)_{reg}$  such that 
$\m P(V)^{\gs h}\neq\emptyset$}.
\end{eqnarray*} 

\begin{question}
\label{conagsorb}
Does one have  the equivalence 
$Rep(\g g, \g h)\Leftrightarrow Orb(\g g, \g h)$ ?
\end{question}
\noindent
We know  that the implication $\Longrightarrow$ is true.\\
We also know that
the converse $\Longleftarrow$ is true
when $\g h$ is reductive. 

\subsubsection{Parabolic induction of tempered representation}
The strategy we followed in this series of paper could be simplified
if we knew the answer to the following

\begin{conjecture} 
	\label{conjparabolic}
	Let $G$ be a real algebraic semisimple group, $Q=LU$ be a parabolic subgroup, 
	and $\pi$ be a unitary representation of $Q$. Does one have
	$$
	\mbox{\rm $\pi$ is $L$-tempered}
	\Longleftrightarrow
	\mbox{\rm ${\rm Ind}_Q^G \pi$ is $G$-tempered.} 
	$$ 
\end{conjecture}
\noindent
We know that the implication $\Longleftarrow$ is true.\\
We have seen  the implication $\Longrightarrow$ 
when $\pi|_U$ is trivial  in Lemma \ref{lemparsub2}.\\ 
We know the implication $\Longrightarrow$ 
when $G={\rm SL}(n,\m R)$ and ${\rm SL}(n,\m C)$.


{\small
	\noindent
	Y. \textsc{Benoist} :
	CNRS-Universit\'e Paris-Saclay, Orsay, France\newline
	e-mail : \texttt{yves.benoist@u-psud.fr}
	
	\medskip
	\noindent
	T. \textsc{Kobayashi} :
	Graduate School of Mathematical Sciences and 
	Kavli IPMU (WPI), The University of Tokyo, Komaba,  Japan\newline
	e-mail : \texttt{toshi@ms.u-tokyo.ac.jp}
}

\end{document}